\documentclass[11pt,reqno]{amsart}
\usepackage[american]{babel}
\usepackage{amsmath}
\usepackage{dsfont,mathtools,amssymb}
\usepackage{tikz}
\usepackage{subfigure}

\usepackage{color}
\usepackage{mathrsfs}
\mathtoolsset{showonlyrefs,showmanualtags}

\definecolor{darkblue}{rgb}{0.1,0.1,0.7}

\definecolor{darkred}{rgb}{0.7,0.1,0.1}

\newtheorem{theorem}{Theorem}[section]
\newtheorem{proposition}[theorem]{Proposition}
\newtheorem{lemma}[theorem]{Lemma}
\newtheorem{corollary}[theorem]{Corollary}
\newtheorem{remark}[theorem]{Remark}
\newtheorem{definition}[theorem]{Definition}

\numberwithin{equation}{section}
\numberwithin{figure}{section}

\newcommand{\sumtwo}[2]{\sum_{\substack{#1 \\ #2}}} % sum with 2 lines 
\newcommand{\scalar}[2]{\langle #1 , #2\rangle}

\newcommand{\ind}{\mathbf{1}}

\newcommand{\e}{\varepsilon}

\renewcommand{\rho}{\varrho}
\renewcommand{\phi}{\varphi}

\newcommand{\tmix}{T_{\rm mix}}

\DeclareMathOperator{\var}{Var}
\DeclareMathOperator{\cov}{Cov}

\newcommand{\be}{\begin{equation}}

\newcommand{\cC}{\ensuremath{\mathcal C}} 
 
\newcommand{\cE}{\ensuremath{\mathcal E}} 
\newcommand{\cF}{\ensuremath{\mathcal F}}

\newcommand{\bbF}{{\ensuremath{\mathbb F}} } 
\newcommand{\bbG}{{\ensuremath{\mathbb G}} }

\newcommand{\bbN}{{\ensuremath{\mathbb N}} }

\newcommand{\bbR}{{\ensuremath{\mathbb R}} }

\newcommand{\bbZ}{{\ensuremath{\mathbb Z}} }

\newcommand{\si}{\sigma} 
\newcommand{\ent}{{\rm Ent} } 

\newcommand{\scr}{\mathscr}

\let\a=\alpha \let\b=\beta   \let\d=\delta  \let\e=\varepsilon
 \let\g=\gamma       \let\l=\lambda
            
\let\r=\rho   \let\t=\tau   
  
\let\D=\Delta   \let\G=\Gamma  \let\L=\Lambda 
\let\O=\Omega

\def\({\left(}
\def\){\right)}
\title{Block factorization of the relative entropy via spatial mixing}
\author{Pietro Caputo}
\address{Department of Mathematics and Physics, Roma Tre University, Largo San Murialdo 1, 00146 Roma, Italy.}
\email{pietro.caputo@uniroma3.it}
\author{Daniel Parisi}
\address{Department of Mathematics and Physics, Roma Tre University, Largo San Murialdo 1, 00146 Roma, Italy.}
\email{daniel.parisi@uniroma3.it}

\begin{document}
\begin{abstract}
We consider spin systems in the $d$-dimensional lattice $\bbZ^d$ satisfying the so-called strong spatial mixing condition.  
We show that the relative entropy functional %with respect to 
of the corresponding  Gibbs measure satisfies a family of inequalities which %allow us to 
control the entropy on a given region $V\subset\bbZ^d$ in terms of a weighted sum of the entropies on blocks $A\subset V$ when each $A$ is given an arbitrary nonnegative weight $\a_A$.  
These inequalities generalize the well known logarithmic Sobolev inequality for the Glauber dynamics.
% and the associated tensorization statements. 
Moreover, they provide a natural extension of the classical Shearer inequality satisfied by the Shannon entropy. 
Finally, 
they imply a family of modified logarithmic Sobolev inequalities which give quantitative  control on the convergence to equilibrium of arbitrary weighted block dynamics of heat bath type. 
\end{abstract}
\maketitle
% \subjclass[2000]{Primary 60K35; secondary 82B20; 82C26.}
\thispagestyle{empty}

%\author{Pietro Caputo and Daniel Parisi}
%\address{Georg Menz\\ Stanford University}
%\email{gmenz@stanford.edu}
%

\section{Introduction}\label{sec:1}
Functional inequalities such as the Poincar\'e and the logarithmic Sobolev inequality have long played a key role in the analysis of convergence to equilibrium for spin systems.  
%has been extensively analyzed in the last forty years.
%The convergence to equilibrium for lattice spin systems in the high temperature regime has been extensively analyzed in the last forty years. One of the 
For the Glauber dynamics associated to a lattice Gibbs measures in the high temperature regime, rather conclusive results were %already 
obtained around thirty years ago in a series of influential papers %in the early nineties 
\cite{HS,Zeg,StrZeg,StrZeg_ct,Lu-Yau,MO2}. Broadly speaking, the main results of these works can be summarized with the statement that for finite or compact spin space, if the spin system satisfies a spatial mixing condition, then the {\em relative entropy} functional of the Gibbs measure $\mu_V$ describing the system on any %spin system in a 
region $V\subset \bbZ^d$, satisfies an {\em approximate tensorization} of the form:
\begin{align}\label{ATC}
\ent_Vf\leq C\sum_{x\in V} \mu_V\!\left[\ent_xf\right],
\end{align}
where $C>0$ is a constant, $f$ is a nonnegative function, and 
$\ent_V f$, the relative entropy of the probability measure $\left(f/\mu_V f\right)\mu_V$
%$fd\mu_V /\mu_V f$ 
with respect to $\mu_V$,  %functional 
is given by 
$$\ent_Vf= \mu_V\left[f\log \left(f/\mu_V f\right)\right],$$
with $\ent_x$ denoting $\ent_{\{x\}}$, for any vertex $x\in V$. The key feature of this inequality is its dimensionless character, namely the fact that  the constant $C>0$ is independent of both the region $V$, and the boundary condition fixed in $\bbZ^d\setminus V$, which we have omitted from our notation for simplicity. The papers mentioned above formulate their results in terms of {\em logarithmic Sobolev inequalities}, but we find it natural to restate them in terms of the tensorization inequality \eqref{ATC}, which seems to have a more fundamental character in our setting. Anyhow, if the spin space is finite,
%The approximate tensorization 
the statement \eqref{ATC} is %always 
equivalent to 
%always implied by 
the standard logarithmic Sobolev inequality for the single site heat bath Markov chain, 
%and it is equivalent to it 
 see e.g.\ \cite{CMT,Marton2}.
 %, which investigate the validity of \eqref{ATC} for weakly dependent spins systems on arbitrary graphs.

The proof of these results % statement \eqref{ACT} 
was obtained through refined recursive techniques, which exploit the spatial mixing assumption to establish some form of factorization %property 
of the entropy functional.   %allowing the use of spatial mixing. 
We refer to the surveys 
\cite{MarStFlour,GuiZeg} for systematic expositions of these techniques. 
A particularly simple and effective approach was later developed in \cite{Cesi} and \cite{PPP}, who
independently  showed that the spatial mixing condition implies a factorization estimate of the form
\begin{align}\label{FC1}
\ent_Vf\leq (1+\e)\,\mu_V\!\left[\ent_Af + \ent_B f\right]\,,
\end{align}
where  
$A,B$ are e.g.\ two overlapping rectangular regions in $\bbZ^d$, with $V=A\cup B$, and $\e>0$ is a constant that can be made small provided the overlap between $A$ and $B$ is sufficiently thick. If the inequality \eqref{FC1} is available, then a relatively simple recursion %represents the main step of a recursive scheme which implies 
leads to the desired conclusion \eqref{ATC}. 

The spatial mixing assumed for all these results is a condition of the Dobrushin-Shlosman type \cite{dobrushin_shlosman}, that can be formulated in terms of exponential decay of correlations. In the literature one finds various degrees of generality of the mixing condition, often loosely referred to as {\em strong spatial mixing}. We refer to the original papers for the precise notions of spatial mixing involved; see also Section \ref{sec:mixing} below for more on this matter. 
%Let us however recall that for certain systems, such as for instance the two-dimensional ferromagnetic Ising/Potts model, the estimate \eqref{ATC} is known to hold throughout the whole one phase region, that is as soon as the system has a unique infinite volume Gibbs measure, provided the region $V$ in \eqref{ATC} is sufficiently regular \cite{MO1,MOS}. 
We point out that the discussion here is mostly concerned with %limited to 
the case of finite %discrete 
or compact spin space, in which case one can actually show that \eqref{ATC} is {\em equivalent} to a strong mixing condition \cite{StroockZeg_eq,MO2}. In the case of unbounded %continuous 
spins the  techniques and the results are somewhat different; we refer the interested reader to \cite{Zeg96,Yoshida,BodHel,Ledoux,OttoRez,Marton}.

While the inequality \eqref{ATC} is well suited for the analysis of the single site heat bath Markov chain, it is not very helpful in the analysis of more general block dynamics, that is Markov chains where an entire region $A\subset V$ can be resampled at once by a single heat bath move.
With that motivation in mind, 
in this work we address the question of the validity of a version of the inequality \eqref{ATC} where single sites $x\in V$ are replaced by arbitrary blocks $A\subset V$. More precisely, we consider the question of finding the best constant $C$ such that for all nonnegative functions $f$,
\begin{align}\label{BTC}
\g(\a)\,\ent_Vf\leq C
%\frac{C}{\g(\a)}
\sum_{A\subset V} \a_A\,
\mu_V\!\left[\ent_Af\right]\,,
\end{align}
where %$C>0$ is a constant, 
$\a=\{\a_A,\,A\subset V\}$  is an arbitrary collection of nonnegative weights, and we define
 \begin{align}\label{BTC1}
\g(\a) =\min_{x\in V} \sum_{A:\,A\ni x}\a_A\,.
\end{align}
If \eqref{BTC} holds with the same constant $C$ for all finite regions $V\subset \bbZ^d$, for all given boundary conditions on $\bbZ^d\setminus V$, and {\em for all choices of weights} $\a$, we say that the spin system satisfies the {\em block factorization of entropy} (with constant $C$).

This definition is inspired by %natural in view of 
the %known 
fact that  
in the case of infinite temperature, that is if $\mu_V$ is a {\em product measure}, then 
\eqref{BTC} holds with  $C=1$.
%\begin{align}\label{BTC1}
%C(\a)= \frac1{\g(\a)}\,,\qquad \g(\a) =\min_{x\in V} \sum_{A:\,A\ni x}\a_A\,.
%\end{align}
Indeed, in this special case it is a consequence of the well known {\em Shearer inequality} satisfied by the Shannon entropy, see \cite{CMT}. These inequalities have far reaching applications in several different settings, see e.g.\ \cite{MadTet,BolloBal,virag_exp}, and it is thus very natural to investigate their validity beyond the product case.

However, as far as we know there are no significant results in the literature concerning the validity of \eqref{BTC} when $\mu_V$ is not a product measure. %at any finite temperature. 
Notice that the tensorization statement \eqref{ATC} corresponds to the special case where $\a_A=1$ or $0$ according to whether $A$ is a single site or not. In this case, the right hand side of \eqref{BTC} has a simple additive structure, a feature that is crucially used in  all existing proofs of \eqref{ATC}.

An important progress was obtained recently in \cite{BCSV} concerning the linearized version of \eqref{BTC}. Namely, if we replace the entropy functional $\ent_V f$ by the {\em variance} functional
 \begin{align}\label{BCSV1}
\var_V\! f = \mu_V\!\left[(f-\mu_V f)^2\right],
\end{align}
then \eqref{BTC} becomes the Poincar\'e inequality
\begin{align}\label{BTCvar}
\g(\a)\var_V\! f\leq C
%\frac{C}{\g(\a)}
\sum_{A\subset V} \a_A\,
\mu_V\!\left[\var_A\!f\right],
\end{align}
which we may refer to as the {\em block factorization of variance}. Notice that the inequality \eqref{BTCvar} 
provides the lower bound $\g(\a)/C$ on the spectral gap of the {\em $\a$-weighted block dynamics}, that is the  Markov chain with Dirichlet form defined by
\begin{align}\label{Diri}
\cE_{V,\a}(f,g)=\sum_{A\subset V} \a_A\,
\mu_V\!\left[\cov_A(f,g)\right],
%\qquad \cov_A(f,g)= \mu_A\left[fg\right]- \mu_A\left[f\right] \mu_A\left[g\right],
\end{align}
where $\cov_A(f,g)= \mu_A\left[fg\right]- \mu_A\left[f\right] \mu_A\left[g\right]$ denotes the covariance of two functions $f,g$ with respect to $\mu_A$. This is the continuous time Markov chain where each block $A$ independently undergoes full heat bath resamplings at the arrival times of a Poisson process with rate $\a_A\geq 0$, see e.g.\ \cite{MarStFlour}.
 
One of the main results of \cite{BCSV} shows that, if the system satisfies the strong spatial mixing assumption, then it must satisfy %display the block factorization of variance 
the special case of \eqref{BTCvar} where the weights $\a$ are all either zero or one, but otherwise arbitrary,  and where $\g(\a)$ is replaced by the indicator $\ind_{\g(\a)>0}$, %thafor some constant $C$, 
see \cite[Theorem 1.2]{BCSV}.
% and Section \ref{sec:genblo} below for more details. 
The proofs in \cite{BCSV} however rely crucially on coupling arguments as in \cite{dyer_et_al}, which do not seem to be effective in establishing the stronger statement \eqref{BTC}. %factorization of entropy. 

In this paper we establish the block factorization of entropy, namely the full statement \eqref{BTC}, provided the system satisfies a strong spatial mixing assumption. For instance, it will follow that the block factorization of entropy holds throughout the whole one phase region for the ferromagnetic Ising/Potts models in two dimensions, provided $V$ in \eqref{BTC} is %restricted to be a sufficiently 
a sufficiently regular set in the sense of \cite{MO1}, %that is  $V$ is the union of translates of a sufficiently large cube in $\bbZ^d$, 
see Section \ref{sec:mixing}. 

As a corollary, we obtain estimates on the speed of  convergence to equilibrium of any block dynamics. Indeed, Jensen's inequality shows that, for any $A\subset V\subset \bbZ^d$,
\begin{align}\label{Jensen}
\ent_Af\leq \cov_A(f,\log f),
\end{align}
and therefore \eqref{BTC} implies the following {\em modified logarithmic Sobolev inequality}
for any $\a$-weighted block dynamics:
\begin{align}\label{MLSI}
\g(\a)\,\ent_Vf\leq C\, \cE_{V,\a}(f,\log f).
\end{align}
In particular, the block factorization of entropy implies the exponential decay in time of the relative entropy, with rate at least $\g(\a)/C$, for any $\a$-weighted block dynamics. Moreover, if the spin state is finite the bound \eqref{MLSI} implies the 
upper bound 
%on the mixing time of the $\a$-weighted block dynamics of the form
\begin{align}\label{Jensen2}
\tmix(V,\a)\leq D\,\g(\a)^{-1}\log |V|,
\end{align}
where $|V|$ is the cardinality of the set $V$, $D$ is  some new absolute constant and $\tmix(V,\a)$ denotes the total variation {\em mixing time} of the $\a$-weighted block dynamics.
%continuous time Markov chain with Dirichlet form 
%the dynamics with Dirichlet form $\cE_{V,\a}$ defined in 
%\eqref{Diri}. 
We refer e.g.\ to \cite{DS,BobTet} for the standard background on these  implications. If the spin state is finite it is also possible to use \eqref{BTC} to derive a
standard logarithmic Sobolev inequality for the 
$\a$-weighted block dynamics in the form
 \begin{align}\label{DS1}
\ent_Vf\leq s(\a)\,\, \cE_{V,\a}\!\left(\sqrt f,\sqrt f\right),
\end{align}
with the constant $$s(\a) = D\,\g(\a)^{-1}\!\max_{A:\,\a_A>0} \log(1/\mu_{A,*}),$$ where $D$ is an absolute constant
and $\mu_{A,*}$ is the minimum value attained by the probability measure $\mu_A$, minimized over the choice of the implicit boundary condition in $\bbZ^d\setminus A$.  Indeed, \eqref{DS1} follows immediately from \eqref{BTC} and a standard bound comparing $\ent_A f$ to $\var_A \!\sqrt f $, see \cite[Corollary A.4]{DS}. 

We conclude this introduction with a brief discussion of the main ideas involved in the proof of our main result \eqref{BTC}. The proof starts with an observation already put forward in \cite{BCSV} for the case of the spectral gap, which allows us to reduce the general factorization problem to the problem of factorization with two special blocks only: the  even sites and the odd sites. The latter is then analyzed via a recursion %techniques 
similar to that employed in Cesi's proof of \eqref{ATC}, see \cite{Cesi}. As mentioned above, the main obstacle in implementing the recursion here is the lack of an additive structure, which generates potentially large error terms when trying to restore a block from smaller components. To overcome this difficulty we develop a two-stage recursion, which combines a version of the two-block factorization estimate \eqref{FC1} together with a %martingale 
decomposition of the entropy which allows us to smear out the errors coming from the restoration of large blocks, see Theorem \ref{th:maind}. A further crucial ingredient in the proof is a new tensorization lemma which we believe to  be of independent interest, see Lemma \ref{lem:parislemma} below.

The plan of the paper is as follows.  In Section \ref{setup} we describe the setup and the main results. 
In Section \ref{sec:parislemma} we develop some key tools needed for the proof. In Section \ref{sec:proof} we prove the block factorization estimate.

\section{Setup and main results}\label{setup}
%We introduce the general setup for the spin systems to be considered.

\subsection{The spin system}
The underlying graph is the $d$-dimensional integer lattice $\bbZ^d$, with vertices $x=(x_1,\dots,x_d)$, and edges $\cE$ defined as %the set of 
unordered pairs $xy$ of vertices $x$ and $y$ such that
$\sum_{i=1}^d|x_i-y_i| =1$. We call $d(\cdot,\cdot)$ the resulting graph distance.  For any set of vertices $\L\subset \bbZ^d$, %we write $|\L|$ for the cardinality, and 
 the exterior boundary is %defined by 
 $\partial \L=\{y\in \L^c:\, d(y,\L)=1\}$, where $\L^c=\bbZ^d\setminus \L$.   We write $\bbF$ for the set of finite subsets $\L\subset\bbZ^d$. 

We take the single spin state to be an arbitrary probability space $(S,\scr{S},\nu)$. 
Given any region $\L\subset \bbZ^d$, the associated configuration space is the product  space $(\O_\L,\cF_\L) = (S^\L,\scr{S}^\L)$, whose elements are denoted by  $\si_\L=\{\si_x,\,x\in \L\}$ with $\si_x\in S$ for all $x$. The {\em apriori measure} on $\O_\L$ is the product measure $\nu_\L=\otimes_{x\in \L}\nu$. %When $V=\bbZ^d$ we simply write $(\O,\cF)$ for the corresponding configuration space. 

Given a {\em bounded} measurable symmetric function $U:S\times S\mapsto \bbR$, the pair potential, and a {\em bounded} measurable function $W:S\mapsto \bbR$, the single site potential, for any $\L\in \bbF$, and $\t\in\O_{\L^c}$, the Hamiltonian $H_\L^\t:\O_\L\mapsto \bbR$ 
%of the spin system in the region $\L$ with boundary condition $\t$ 
is defined by
\begin{align}\label{Ham}
H_\L^\t (\si_\L) = -\sumtwo{xy\in\cE:}{x,y\in \L}U(\si_x,\si_y) - \sumtwo{xy\in\cE:}{x\in \L,y\in \partial \L}U(\si_x,\t_y) - \sum_{x\in \L} W(\si_x). 
\end{align}
The Gibbs measure in the region $\L\in\bbF$  with boundary condition $\t\in\O_{\L^c}$ is the probability measure $\mu_\L^\t$ on $ (\O_\L,\cF_\L)$ defined by 
\begin{align}\label{Gib}
\mu_\L^\t(d\si_\L) = \frac1{Z_\L^\t}\exp\left[-H_\L^\t (\si_\L)\right] \nu_\L(d\si_\L) \,,
\end{align}
where $Z_\L^\t$ 
%= \int \exp\left[-H_\L^\t (\si_\L)\right] \nu_\L(d\si_\L)$ 
is the normalizing constant.

For any %bounded 
measurable function $f:\O_\L\mapsto \bbR$ we write $\mu_\L^\t f $ for the expectation of $f$ under $\mu_\L^\t$, and write $\mu_\L f$ for the measurable function $\O_{\L^c}\ni\t\mapsto \mu_\L^\t f$. 
A fundamental feature of the family of measures $\{\mu_\L^\t,\,\L\in\bbF\,,\t\in\O_{\L^c}\}$ is the so-called DLR property:
\begin{align}\label{DLR}
\mu_V%\left(
\mu_\L f%\right)
 = \mu_V f\,,
\end{align}
valid for all $\L\subset V\in\bbF$, and for all bounded measurable function $f:\O_V\mapsto \bbR$.

%
%
%is that for any $\D\subset\L\in\bbF $, for any bounded measurable $f:\O_\L\mapsto \bbR$, the function $\mu_\D f$ is a version of the conditional expectation $\mu_\ 
%, where for u, v ? Zd, (u, v) ? E iff ||u ? v||1 = 1. Let V be a finite subset of Zd and let G = (V, E) be the induced subgraph. We use ?V to denote the boundary of G, i.e., the set of vertices in Zd \ V connected by an edge in E to V .
\subsection{Examples and remarks}
Below we list some standard examples which fit the general framework defined above and discuss possible extensions. 
We refer the reader to \cite{FriedliVelenik} for an introduction to the statistical mechanics of lattice spin systems.  

\subsubsection{Finite spins}
When the space $S$ is finite we take $\nu$ as the counting measure on $S$. The {\em Potts model} corresponds to  $S=\{1,\dots,q\}$, with $q\geq 2$ a fixed integer, $$U(s,s')=\b\,\ind_{\{s=s'\}}\,,\qquad W(s)=\b\, h_s,$$ where the parameter $\b\in\bbR$ is related to the inverse temperature of the system and the fixed vector $(h_1,\dots,h_q)\in \bbR^q$ to an external magnetic field.  When $\b\geq0$ the model is called ferromagnetic. When $q=2$ the Potts model is called the {\em Ising model}.  
In the case of finite spin space, in order to include spin systems with hard constraints, we shall also allow the function $U$ to take the value $-\infty$. The spin system is called {\em permissive} if  for every $\L\in\bbF$, for every $\t\in \O_{\L^c}$, there exists $\si_\L\in\O_\L$ with positive mass under $\mu_\L^\t$, that is such that $\mu_\L^\t(\si_\L) >0$. Well known examples of permissive spin systems include the {\em hard-core model} with parameter $\l$, for any $\l>0$, and the uniform distribution over {\em proper $q$-colorings}, 
for any integer $q\geq 2d+1$. 
The hard-core model with parameter $\l$ corresponds to $S=\{0,1\}$, $U(1,1)=-\infty$, $U(1,0)=U(0,1)=U(0,0)=0$, $W(s)=s\log(\l)$, while the uniform distribution over proper $q$-colorings corresponds to the limit $\b\to-\infty$ in the Potts model. A permissive spin system is called {\em irreducible} if the single site heat bath Markov chain on $\L$ with boundary condition $\t$ is irreducible for any choice of $\L\in\bbF$ and $\t\in\O_{\L^c}$, see \cite[Section 2]{BCSV}. 
Our main results below will apply to permissive irreducible spin systems.
%this setting, provided we additionally require that the single site heat bath Markov chain on $\L$ with boundary condition $\t$ is irreducible for any choice of $\L\in\bbF$ and $\t\in\O_{\L^c}$, see \cite[Section 2]{BCSV}. In this case we say that the permissive system is irreducible.

\subsubsection{Continuous compact spins}
Other classical examples are obtained when $S$ is a compact subset of $\bbR^n$ and $\nu$ is the uniform distribution over $S$. The {\em {\em O}$(n)$ model}, for $n\geq 2$, corresponds to the case where $S$ is the unit sphere in $\bbR^n$, $\b\in\bbR$,  $$U(s,s')=\b\scalar{s}{s'}\,,\qquad W(s)=\b\scalar{s}{v}\,,$$ for some fixed vector $v\in S$, with $\scalar{\cdot}{\cdot}$ denoting the standard inner product in $\bbR^n$. 

\subsubsection{Unbounded spins}
The setup introduced above includes unbounded (continuous or discrete) spins. When $S=\bbZ_+$ for instance it covers the particle systems considered in \cite{PPP}. It should be however clear that the boundedness assumptions on the interaction $U$ rules out many interesting models in the unbounded setting.  

%To keep things simple we have restricted o

\subsubsection{Extensions}
Concerning possible extensions of our main results to more general settings, we remark that the definitions given above can be extended to include spatially non-homogeneous models, with pair potentials $U$ and site potentials $W$ replaced by edge dependent functions $U_{xy}$ and site dependent functions $W_x$ respectively. It is not difficult to check that all results in this paper can be extended to include these cases provided that all the estimates involved in our assumptions are uniform with respect to the new potentials. 
Finally, we remark that our setup is restricted to the case of nearest neighbor interactions, and  
the extension of our main results to more general finite range spin systems is not immediate. Indeed, our proof makes explicit use of the nearest neighbor structure at various places. We believe however that a similar approach can be used, provided the decomposition into even and odd sites used in our proof is replaced by more general tilings such as the ones used in \cite{BCSV}.  

%%We believe that our result continue to hold for
% We refer to Remark \ref{rem:finiterange} below for possible extensions of our results to more general finite range spin systems. Finally, we remark that the definitions given above can be extended to include spatially non-homogeneous models, with pair potentials $U$ and site potentials $W$ replaced by edge dependent functions $U_{xy}$ and site dependent functions $W_x$ respectively. It is not difficult to check that all results in this paper can be extended to include this cases provided that all assumptions made involve estimates that are uniform in these new potentials. 
%
% 
%
%\pietro{Aggiustare sopra }
 
\subsection{Spatial mixing} \label{sec:mixing}
The notion of spatial mixing to be considered 
%coincides with the one adopted in Cesi's paper \cite{Cesi}. It 
belongs to the family of %so-called 
{\em strong spatial  mixing} conditions. In the case of finite spins %systems 
it is one of many equivalent %to a whole set of 
conditions introduced by Dobrushin and Shlosman \cite{dobrushin_shlosman} to characterize the so-called {\em complete analyticity} regime.

The precise formulation we give here coincides with the one adopted in Cesi's paper \cite{Cesi}.
For any $\D\subset \L\in\bbF$ we call $\mu_{\L,\D}^\t$ the marginal of $\mu_{\L}^\t$ on $\O_\D$. A version of the  
Radon-Nikodym density of $\mu_{\L,\D}^\t$ with respect to $\nu_\D$ is given by the function
\begin{align}\label{radon-nikodym}
\psi_{\L,\D}^\t(\si_\D):=\frac1{Z_\L^\t}\int\exp\left[-H_\L^\t (\eta_{\L\setminus\D}\si_\D)\right] \nu_{\L\setminus\D}(d\eta_{\L\setminus\D}),
\end{align}
where $\eta_{\L\setminus\D}\si_\D$ denotes the configuration  $\xi\in\O_\L$ such that $\xi_x=\eta_x$ if $x\in\L\setminus\D$ and $\xi_x=\si_x$ if $x\in\D$. 

\begin{definition}\label{def:ssm}
Given constants $K,a>0$, and $\L\in\bbF$ we say that condition $\cC(\L,K,a)$ holds if 
for any %choice of finite subsets 
$\D\subset \L$,  for all $x\in \partial\L$:
\begin{align}\label{ssmca}
\sup_{\t,\t'}\, %\sup_{\si_\D}
\left\|\,\frac{\psi_{\L,\D}^{\t'}}{\psi_{\L,\D}^{\t}}\,-\,1\,\right\|_\infty\leq \,K\,e^{-a\,d(x,\D)},
\end{align}
where % the supremum is over  %spin configurations $\si\in \O_\D$, while the
%the spin configurations 
$\t,\t'\in\O_{\L^c}$ are such that $\t_y=\t'_y$ for all $y\neq x$, and $\|\cdot\|_\infty$ denotes the $L^\infty$ norm. We say that the spin system satisfies $SM(K,a)$ if $\cC(\L,K,a)$ holds for all $\L\in\bbF$. 
  \end{definition}

As emphasized in \cite{MO1} it is often %sometimes 
 important to consider a relaxed spatial mixing condition that requires $\cC(\L,K,a)$ to hold only for all sufficiently ``fat" sets $\L$. The latter is defined as follows. 
 \begin{definition}\label{def:ssmreg}
 Given $L\in\bbN$, let $Q_L=[0,L-1]^d\cap\bbZ^d$ be the lattice cube of side $L$ located  at the origin. 
 For any $y\in\bbZ^d$, define the translated cube $Q_L(y)=Ly + Q_L$. 
Let $\bbF^{(L)}$ be the set of all $\L\in\bbF$ of the form 
$$
\L=\bigcup_{y\in \L'} Q_L(y)\,
$$
 for some $\L'\in\bbF$. 
The spin system satisfies $SM_L(K,a)$ if $\cC(\L,K,a)$ holds for all $\L\in\bbF^{(L)}$. 
 \end{definition} 
 
For systems without hard constraints it is well known that $SM(K,a)$, for some $K,a$, is always satisfied in dimension one, and that for any dimension $d>1$ it holds under the assumption of suitably high temperature,
%a suitably small interaction
see e.g.\ \cite{MarStFlour}. % that any one-dimensional  
It is important to note that the validity of both %properties 
$SM(K,a)$ and $SM_L(K,a)$ can be ensured by checking  finite size conditions only \cite{Fabio_hx}.
 
 We recall that %the condition 
 $SM(K,a)$ can be strictly stronger than $SM_L(K,a)$. 
For instance, as a consequence of results in \cite{MOS,Alexander,BeffaraDuminil} it is known that the two-dimensional ferromagnetic Potts model satisfies $SM_L(K,a)$, for some $K,a>0$ and  $L\in\bbN$, throughout the whole uniqueness region, while $SM(K,a)$ cannot hold in this generality.

 %all the way up to the critical point.   

 %We shall prove our main results under the assumption that the system satisfies $SM(K,a)$ for some $K,a>0$, but it is not hard to extend them to include this relaxed spatial mixing condition, see Remark \ref{rem:regular}.
Finally, we  note that $\cC(\L,K,a)$ %\eqref{ssmca} 
is too strong a requirement in the case of systems with hard constraints, since 
$\mu_{\L,\D}^{\t'}$ may be not absolutely continuous with respect to $\mu_{\L,\D}^\t$.
% may be singular if e.g.\ the vertex $x\in\partial\L$ is very close to $\D$. 
However, since \eqref{ssmca} will only be relevant if $d(x,\D)$ is sufficiently large, in order to have a meaningful assumption for permissive spin systems with hard constraints, we may rephrase the condition $SM_L(K,a)$ by requiring, for all $\L\in\bbF^{(L)}$, that \eqref{ssmca} holds for all $\D\subset \L$ and $x\in \partial \L$ such that $d(x,\D)\geq L/2$.

\subsection{Main results}
We first recall some standard notation. For any $V\in\bbF$, $\t\in\O_{V^c}$, and $f:\O_V\mapsto\bbR_+$ with $f\log^+\!f\in L^1(\mu_V^\t)$, we write $\ent_V^\t f$ for the entropy 
 \begin{align}\label{ent}
\ent_V^\t f = \mu_V^\t\left[f\log\left(f/\mu_V^\t f\right)\right],
\end{align}
and use the notation $\ent_V f$ for the function $\t\mapsto \ent_V^\t f$. 
%We adopt the convention that whenever $\ent_V f$ appears in some inequality it is understood that the inequality holds uniformly over $\t$ for $ \ent_V^\t f$.
\begin{theorem}\label{th:mainresult}
Suppose that the spin system satisfies $SM(K,a)$ for some constants $K,a>0$. Then there exists a constant $C>0$ such that for all $V\in\bbF$, $\t\in \O_{V^c}$, for all nonnegative weights $\a=\{\a_A,\,A\subset V\}$,
for all %functions 
$f:\O_V\mapsto\bbR_+$ with $f\log^+\!f\in L^1(\mu_V^\t)$,
\begin{align}\label{BTCt}
\g(\a)\,\ent_V^\t f\leq C
\sum_{A\subset V} \a_A\,
\mu_V^\t\!\left[\ent_Af\right],
\end{align}
where $\g(\a) =\min_{x\in V} \sum_{A:\,A\ni x}\a_A$.
If instead the spin system satisfies $SM_L(K,a)$ for some constants $K,a>0$, $L\in\bbN$, then the conclusion \eqref{BTCt} continues to hold, provided we require that $V\in\bbF^{(L)}$. 
\end{theorem}

As we mentioned in Section \ref{sec:1}, Theorem \ref{th:mainresult} has the following immediate corollary for the $\a$-weighted block dynamics defined by \eqref{Diri}. Below, $\cE^\t_{V,\a}(f,g)$ denotes the Dirichlet form \eqref{Diri} evaluated at a given boundary condition $\t\in\O_{V^c}$.  %We refer to \cite{DS,BobTet} for the standard
\begin{corollary}\label{corolla}
If the spin system satisfies $SM(K,a)$ for some constants $K,a>0$, then the following modified logarithmic Sobolev inequalities hold: 
for all $V\in\bbF$, all $\t\in \O_{V^c}$, for all weights $\a$,
for all $f:\O_V\mapsto\bbR_+$ with $f\log^+\!f\in L^1(\mu_V^\t)$,
\begin{align}\label{MLSIa}
\g(\a)\,\ent_V^\t f\leq C\,\cE^\t_{V,\a}(f,\log f),
\end{align}
where $\g(\a)$ and $C$ are the same constants appearing in \eqref{BTCt}.
In particular, if the spin state $S$ is finite, then there exists a constant $D>0$ such that for all $V\in\bbF$, $\t\in \O_{V^c}$, for all weights $\a$, the %total variation 
mixing time $\tmix^\t(V,\a)$ of the Markov chain with Dirichlet form $ \cE^\t_{V,\a}$ satisfies
\begin{align}\label{tmix2}
\tmix^\t(V,\a)\leq D\,\g(\a)^{-1}\log |V|.
\end{align}
Moreover, if the spin state is finite, then $SM(K,a)$ implies the following logarithmic Sobolev inequalities: 
there exists a constant  $D>0$ such that for all $V\in\bbF$, all $\t\in \O_{V^c}$, for all weights $\a$, all $f\geq 0$
with $f\log^+\!f\in L^1(\mu_V^\t)$,
 \begin{align}\label{DS1a}
\ent^\t_Vf\leq s(\a)\, \,\cE^\t_{V,\a}\!\left(\sqrt f,\sqrt f\right),\qquad s(\a) = D\,\g(\a)^{-1}\!\max_{A:\,\a_A>0} \log(1/\mu_{A,*}),
\end{align}
where $$\mu_{A,*} = \min_{\t\in\O_{A^c}}\min_{\si_A\in\O_A:\, \mu_A^\t(\si_A)>0}\mu_A^\t(\si_A)\,.$$
Finally, all statements above continue to hold if we only assume $SM_L(K,a)$ for some constants $K,a>0$ and $L\in\bbN$, provided we restrict to $V\in\bbF^{(L)}$. 
 \end{corollary}

\section{Some key  tools}\label{sec:parislemma}
In this section we collect some key general facts that do not depend on the spatial mixing assumption. 
We start by recalling some standard decompositions of the entropy.  Next, we prove a new general tensorization lemma. 
Finally, we revisit the two-block factorization \eqref{FC1}. 

%In this section w
Some remarks on the notation are in order. We fix a region $V\in\bbF$ and a boundary condition $\t\in\O_{V^c}$. To avoid heavy notation, we often omit explicit reference to $V,\t$. In particular, whenever possible we shall use  the following shorthand notation
\begin{align}\label{short}
\mu f = \mu_V^\t f\,,\qquad \ent f = \ent_V^\t f\,.
\end{align}
Moreover, whenever we write $\mu_\L$ or $\ent_\L$ for some $\L\subset V$, we assume that the implicit boundary condition outside $\L$ has been fixed, and it agrees with $\t$ outside of $V$. %It will also be convenient to 
%adopt the convention that whenever $\ent_\L f$ appears in some inequality it is understood that the inequality holds uniformly over the implicit boundary condition. 
Unless otherwise stated, $f$ will always denote a nonnegative measurable function such that $f\log^+ f\in L^1(\mu)$. 
To avoid repetitions, we simply write $f\geq 0$ throughout. As a convention, we set $\mu_\emptyset f = f$ and $\ent_\emptyset f=0$.

\subsection{Preliminaries} 
We first recall a standard lemma that will be repeatedly used. 
\begin{lemma}\label{lem:telescope}
For any $\L\subset V$, for any $f\geq 0$:
\begin{align}\label{deco}
\ent f = \mu\left[\ent_{\L}f\right]
+ \ent\,\mu_\L f.
\end{align}
More generally, for any %finite monotone set sequence 
$\L_0\subset \L_1\subset\cdots\subset \L_k\subset V$, for any $f\geq 0$:
\begin{align}\label{tele}
\sum_{i=1}^{k}\mu\left[\ent_{\L_{i}} \mu_{\L_{i-1}}f\right]
=\mu\left[\ent_{\L_k}\mu_{\L_0}f\right].
\end{align}
\end{lemma}
\begin{proof}
The identity \eqref{deco} follows from \eqref{tele} in the case $k=2$ with $\L_0=\emptyset$, $\L_1=\L$, $\L_2=V$. 
%\begin{align}
%\ent f &= \mu\left[f\log(f/\mu f)\right] \\&
%= \mu\left[f\log(f/\mu_\L f)\right] + \mu\left[f\log(\mu_\L f/\mu f)\right]\\& 
%= \mu\left[\mu_\L\left(f\log(f/\mu_\L f)\right)\right] + \mu\left[\mu_\L f \log(\mu_\L f/\mu f)\right]
%\\& =  \mu\left[\ent_{\L}f\right]
%+ \ent\left(\mu_\L f\right).
%\end{align}
To prove \eqref{tele}, set %$g=\mu_{\L_0}f$, 
$g_i=\mu_{\L_i}f$, and note that $g_{i}=\mu_{\L_{i}}g_{i-1}$ by \eqref{DLR}. Therefore,
\begin{align}\label{tele1}
\mu\left[\ent_{\L_k}g_0\right] &= \mu\left[g_0
\log\left(g_0/g_k\right)\right]
% = \mu\left[g_0\log g_0\right] - \mu\left[g_k\log g_k\right]
\\
&
=\sum_{i=1}^{k}  \mu\left[g_{i-1}\log \left(g_{i-1}/g_{i}\right)\right]
\\
&
=\sum_{i=1}^{k}  \mu\left[\mu_{\L_i}\left(g_{i-1}\log \left(g_{i-1}/g_{i}\right)\right)\right]
=\sum_{i=1}^{k}\mu\left[\ent_{\L_{i}} g_{i-1}\right].
\end{align}
\end{proof}

\subsection{A new tensorization lemma}
Consider subsets $$A_{i,j}\subset V\in\bbF \,, \qquad i=1,\dots,n\,,\;j=1,\dots,m,$$ such that $\cup_{i,j}A_{i,j}=\L\subset V$,  and define  ``row'' subsets  and ``column'' subsets:
$$R_i:=\cup_{j=1}^m A_{i,j}\,,\qquad C_j:=\cup_{i=1}^n A_{i,j}\,.$$
Assume that $\mu_\L$ is a product measure along the partition $\{R_i,\,i=1\dots,n\}$ of $\L$:
$$
\mu_\L=\otimes_{i=1}^n\mu_{R_i}\,.
$$
Notice that this is the case if $\{R_1,\dots,R_n\}$
are %subsets of $\bbZ^d$ satisfying 
such that $d(R_i,R_j)>1$ for all $i\neq j$. 
\begin{lemma}\label{lem:parislemma}
Let $s_i>0$ be constants such that 
for each $i=1,\dots,n$, for all $f\geq 0$, %S^V\mapsto\bbR_+$,
\begin{equation}\label{paris2}
\ent_{R_i}f \leq  s_i \sum_{j=1}^m\mu_{R_i}[\ent_{A_{i,j}}f].
\end{equation}
Then
\begin{equation}\label{paris3}
\ent_\L f\leq  s \sum_{j=1}^m\mu_\L[\ent_{C_j}f],
\end{equation}
where $s=\max_is_i$.
\end{lemma}

\begin{proof}
To simplify the notation, we write $\mu=\mu_\L$ and $\ent_\L f=\ent f$. 
Setting $\L_k = \cup_{i=1}^k R_i$, with $\L_0=\emptyset$, from Lemma \ref{lem:telescope} we have 
%Let $f_k= \mu(f|\si_{R_1},\dots,\si_{R_k})$, so that $f_n=f$ and $f_0 = \mu f$. Therefore,
\begin{equation}\label{paris4}
\ent f= \sum_{k=1}^n \mu\left[\ent_{\L_{k}}\mu_{\L_{k-1}}f\right].
%\sum_{i=1}^n\mu[f_i\log f_i\, -\, f_{i-1}\log f_{i-1}].
\end{equation}
Since $\mu_{\L_k}$ is a product of $\mu_{R_i}$, $i=1,\dots,k$, we have %$f_{k-1}=\mu_{R_k}[f_k]$, and 
\begin{equation}\label{paris5}
\ent f= \sum_{k=1}^n \mu\left[\ent_{R_{k}}\mu_{\L_{k-1}}f\right].
%  \sum_{i=1}^n\mu[\ent_{R_i}f_i].
\end{equation}
From \eqref{paris2} we estimate
\begin{equation}\label{paris6}
\ent f\leq  s \sum_{k=1}^n\sum_{j=1}^m\mu[\ent_{A_{k,j}}\mu_{\L_{k-1}}f].
\end{equation}
The proof is complete once we show that for each $j$,
\begin{equation}\label{paris7}
\sum_{k=1}^n\mu[\ent_{A_{k,j}}\mu_{\L_{k-1}}f]\leq\mu[\ent_{C_j}f]. 
\end{equation}
Define $\L_{k,j}=\L_k \cap C_j$. From Lemma \ref{lem:telescope} we have 
\begin{equation}\label{paris40}
\ent_{C_j} f= \sum_{k=1}^n \mu_{C_j}\left[\ent_{\L_{k,j}}\mu_{\L_{k-1,j}}f\right].
\end{equation}
For each $j,k$ fixed, $\mu_{\L_{k,j}}$ is a product of $\mu_{A_{i,j}}$, $i=1,\dots,k$. Hence, 
\begin{equation}\label{paris41}
\ent_{C_j} f= \sum_{k=1}^n \mu_{C_j}\left[\ent_{A_{k,j}}\mu_{\L_{k-1,j}}f\right].
\end{equation}
Therefore, \eqref{paris7} follows if we show that
all $j,k$ fixed:
\begin{equation}\label{paris42}
\mu[\ent_{A_{k,j}}\mu_{\L_{k-1}}f]\leq  \mu\left[\ent_{A_{k,j}}\mu_{\L_{k-1,j}}f\right].
\end{equation}
To prove \eqref{paris42}, notice that  %$\mu_{\L_{k-1}}f=\mu_{\L_{k-1}}\mu_{\L_{k-1,j}}f$, and that 
\begin{equation}\label{paris43}
\mu_{A_{k,j}}\mu_{\L_{k-1}}f=\mu_{A_{k,j}}\mu_{\L_{k-1}}\mu_{\L_{k-1,j}}f = \mu_{\L_{k-1}}\mu_{A_{k,j}}\mu_{\L_{k-1,j}}f,
\end{equation}
where the second identity follows from the product structure $\mu_{\L_k}=\otimes_{i=1}^k\mu_{R_i}$.
Therefore,
\begin{align}\label{paris44}
\mu\left[\ent_{A_{k,j}}\mu_{\L_{k-1}}f\right]
&= 
\mu\left[\mu_{\L_{k-1}}f\log\left(\mu_{\L_{k-1}}f/\mu_{A_{k,j}}\mu_{\L_{k-1}}f\right)\right]
\\ &=
\mu\left[\mu_{\L_{k-1}}\mu_{\L_{k-1,j}}f\log\left(\mu_{\L_{k-1}}\mu_{\L_{k-1,j}}f/\mu_{\L_{k-1}}\mu_{A_{k,j}}\mu_{\L_{k-1,j}}f
\right)\right]
\\ &=
\mu\left[\mu_{\L_{k-1,j}}f\log\left(\mu_{\L_{k-1}}\mu_{\L_{k-1,j}}f/\mu_{A_{k,j}}\mu_{\L_{k-1}}\mu_{\L_{k-1,j}}f
\right)\right]
\\ &\leq
\mu\left[\mu_{A_{k,j}}\left(\mu_{\L_{k-1,j}}f\log\left(\mu_{\L_{k-1,j}}f/\mu_{A_{k,j}}\mu_{\L_{k-1,j}}f
\right)\right)\right]
\\ &=
\mu\left[\ent_{A_{k,j}}\mu_{\L_{k-1,j}}f\right],
\end{align}
where the inequality follows from the variational principle
\begin{equation}\label{varprin}
\ent_{\,\,\!U}(g) = \sup\{\mu_{\,\,\!U} (g h) \,,\; \mu_{\,\,\!U}(e^h)\leq 1\},
\end{equation}
valid for any region $U$, any boundary condition on $U^c$,  and any function $g\geq0$.
%, applied with $U=A_{k,j}$ and $g=\mu_{\L_{k-1,j}}f$.
\end{proof}

\bigskip

Here is an example to keep in mind, with $n$ arbitrary and $m=2$. Let $\{R_1,\dots,R_n\}$
denote a collection of subsets $R_i\in\bbF$ with $d(R_i,R_j)>1$ for all $i\neq j$. 
Let $A_{i,1}=ER_i$ be the even sites in $R_i$ and  $A_{i,2}=OR_i$ be the odd sites in $R_i$, where a vertex $x\in\bbZ^d$ is even or odd according to the parity of $\sum_{i=1}^dx_i$.
 Lemma \ref{lem:parislemma} says that if we can factorize the even and odd sites on 
each $R_i$ with some constant $s_i$, then we can also factorize, with the constant $\max_i s_i$, the even and odd sites on all $\L=\cup_i R_i$. In this example, one has $A_{i,j}\cap A_{i,k}=\emptyset$ if $k\neq j$, so in particular $C_j\cap C_k=\emptyset$ for $k\neq j$, but it is interesting to note that this need not be the case in Lemma \ref{lem:parislemma}, that is each ``row'' $R_i$ is allowed to be decomposed into arbitrary, possibly overlapping subsets $A_{i,j}$, $j=1,\dots,m$. We refer to Remark \ref{rem:strength} for useful  applications of  the latter  situation.

\subsection{Two block factorizations}
We shall need the following versions of an inequality of Cesi \cite{Cesi}.
%revisit the inequality 
%Suppose that $V\in\bbF$ is the union of two blocks $A,B\in\bbF$, with that is $V=A\cup B$. If $d(A,B)>1$ then $\mu=\mu_A\mu_B$ is a product measure and one has the familiar factorization
%\begin{equation}\label{eq:ABtensor}
%\ent f \leq  \mu[\ent_{A}f+\ent_{B}f],
%\end{equation}
%We shall need the observation by Cesi \cite{Cesi} that if 
%$A\cap B\neq\emptyset$ then the inequality still holds up to a constant larger than 1, which can be made as close to 1 as we wish if $ A\cap B\neq\emptyset$ have a suitably thick overlap and the spin system satisfies strong mixing...
%\pietro{rephrase better}.
%We shall need the following slight modification of an inequality of Cesi.
\begin{lemma}\label{lem:entBA}
Take $A,B\in\bbF$ and $V=A\cup B$. Suppose that for some $\e\in(0,1)$:
\begin{gather}
\|\mu_B\mu_A g - \mu g\|_\infty \leq \e\, \mu (|g|)\,\label{Cesia2}
\end{gather}
for all functions $g\in L^1(\mu)$.
Then, for all functions $f\geq 0$,
\begin{align}\label{Cesi1}
&\ent f\leq \mu[\ent_{A}f + \ent_{B} f] + \theta(\e)\,\ent f,
\\ &
\label{Cesi2}
\ent f\leq \mu[\ent_{A}f + \ent_{B}\mu_A f] + \theta(\e)\,\ent 
\,\mu_Af,
\end{align}
where $\theta(\e)=84\e(1-\e)^{-2}$. 
\end{lemma}
\begin{proof}
The inequality \eqref{Cesi1} coincides with \cite[Eq. (2.10)]{Cesi}. To prove 
\eqref{Cesi2} we use essentially the same argument.
 As in the proof of \eqref{Cesi1} we may restrict to the case where $f$ is bounded, and bounded away from zero. Then 
\begin{align*}
\ent f& %= \mu\left[f\log \left(f/\mu f\right)\right] \\&
=  \mu\left[f\log \left(f/\mu_A f\right)\right] +  
\mu\left[f\log \left(\mu_A f/\mu f\right)\right]
\\& 
= \mu[\ent_{A}f] + \mu\left[\mu_A f\log \left(\mu_A f/\mu f\right)\right]
\\& 
= \mu[\ent_{A}f] + \mu[\ent_{B}\mu_A f] + 
\mu\left[\mu_Af\log \left(\mu_B\mu_A f/\mu f\right)\right].
\end{align*}
Cesi's inequality \cite[Eq. (3.2)]{Cesi} says that the assumption \eqref{Cesia2} implies 
\begin{equation}\label{Cesi11}
\mu\left[f\log \left(\mu_B\mu_A f/\mu f\right)\right]\leq \theta(\e)\,\ent f\,,
\end{equation}
for all $f\geq 0$, where $\theta(\e)=84\e(1-\e)^{-2}$. Therefore, the claim \eqref{Cesi2} follows from \eqref{Cesi11} 
applied with $\mu_Af$ in place of $f$.
\end{proof}

\begin{remark}\label{rem:variation}
If $\mu$ is a product measure over $A,B$, that is  $\mu=\mu_B\mu_A$, then one can take $\e=0$ in Lemma \ref{lem:entBA}.
In this case \eqref{Cesi2} is actually an identity. In this sense \eqref{Cesi2} might be considered to be tighter than \eqref{Cesi1}, although it is not true that $ \mu[\ent_{B}\mu_A f]\leq  \mu[\ent_{B} f]$ in the general non-product case: think for instance of some $f$ which depends only on $A\setminus B$; in this case $ \mu[\ent_{B} f]=0$ while it is possible that $\mu[\ent_{B}\mu_A f]>0$. For our purposes below it will be crucial to use both \eqref{Cesi1} and \eqref{Cesi2}.
\end{remark}

\begin{remark}\label{rem:strength}
To appreciate the strength of the tensorization Lemma \ref{lem:parislemma}, consider a case where $V=\cup_{i=1}^nR_i$ with $R_i=A_i\cup B_i$ and suppose that $\mu_V$ is a product measure over the $R_i$'s. If the condition \eqref{Cesia2} holds for every pair $A_i,B_i$, $i=1,\dots,n$, with the same constant $\e\in(0,1)$, the combination of Lemma \ref{lem:entBA} and Lemma \ref{lem:parislemma} shows that 
%we can obtain e.g.\ 
\eqref{Cesi1} holds uniformly in $n$, with $A=\cup_{i=1}^nA_i$ and $B=\cup_{i=1}^n B_i$. On the other hand,  Lemma \ref{lem:entBA} alone cannot yield such a uniform estimate. Indeed, the assumption \eqref{Cesia2} does not tensorize: it is not hard to construct examples where \eqref{Cesia2} holds for every pair $A_i,B_i$, $i=1,\dots,n$, with the same error $\e\in(0,1)$, but one has to take the error proportional to $n $ in order to have \eqref{Cesia2} for $A=\cup_{i=1}^nA_i$ and $B=\cup_{i=1}^n B_i$. 
\end{remark}
%\pietro{Add here the application using SSM }
%\pietro{To appreciate the power of Lemma \ref{lem:parislemma} consider the case of the bound \ref{eq:ABtensor}... $(1+\e)$..}.

\section{Proof of the main results}\label{sec:proof}
We first reduce the general block factorization problem to the factorization into even and odd sites only.
\subsection{Reduction to even and odd blocks}\label{sec:reduction}
We partition the vertices of $\bbZ^d$ into even sites and odd sites, where $x$ is {\em even} if $\sum_{i=1}^dx_i$ is an even integer, while $x$ is  {\em odd} if $\sum_{i=1}^dx_i$ is an odd integer.
Given a set of vertices $V\in\bbF$ we write $EV$ for the set of even vertices $x\in V$ and $OV$ for the set of odd vertices $x\in V$. Whenever possible we simply write $E$ for $EV$ and $O$ for $OV$. Notice that both $\mu_E$ and $\mu_O$ are {\em product measures}. 

The reduction to even and odd blocks can be stated as follows. As usual we assume that a region $V\in\bbF$, and a boundary condition $\t\in\O_{V^c}$ have been fixed, and we use the shorthand notation \eqref{short}.  
\begin{proposition}\label{prop:reduction}
Suppose that for some constant $C>0$ and some function $f\geq 0$, % one has
\begin{equation}\label{eofacto}
\ent f\leq C\,\mu\left[\ent_E f + \ent_O f\right].
\end{equation}
Then, for the same $C$ and $f$, for all nonnegative weights $\a=\{\a_A,\,A\subset V\}$,
\begin{equation}\label{eofacto1}
\g(\a)\,\ent f\leq 2\,C\sum_{A\subset V}\a_A
\, 
\mu\left[\ent_A f\right]
,
\end{equation}
where $\g(\a)=\min_{x\in V}\sum_{A: A\ni x}\a_A$. 
\end{proposition}
Proposition \ref{prop:reduction} is a direct consequence of the following version of Shearer's inequality satisfied by the relative entropy functional of any product measure. 
\begin{lemma}\label{lem:shearer}
Fix $\L\subset V\in \bbF$ and suppose that $\mu_\L$ is a product measure on $\O_\L$. Then, 
for any choice of nonnegative weights $ \a=\{ \a_A, A\subset \L\}$ and any function $f\geq 0$:
\begin{equation}\label{shearer}
\g( \a)\,\ent_\L f\leq \sum_{A\subset \L}\a_A\, \mu_\L\!\left[\ent_A f\right],
\end{equation}
where $\g( \a)=\min_{x\in \L}\sum_{A: A\ni x} \a_A$. 
\end{lemma}
\begin{proof}
As in \cite[Proposition 2.6]{CMT}, the inequality \eqref{shearer} follows from a weighted version of Shearer's inequality for Shannon entropy. For a proof of the latter we refer e.g.\ to \cite[Theorem 6.2]{virag_exp}.
\end{proof}
\begin{proof}[Proof of Proposition \ref{prop:reduction}]
Fix a choice of weights $\a=\{\a_A, A\subset V\}$. Since $\mu_E$ is a product measure on $\O_E$, 
%and $\mu_O$ are product measures, 
we may apply Lemma \ref{lem:shearer} with $\L=E$ and weights $\a$ replaced by $\hat \a=\{\hat \a_U,\,U\subset E\}$, with $\hat \a_U = \sum_{A\subset V}\a_A\ind_{EA=U}$. It follows that
\begin{align}\label{shear10}
&\sum_{A\subset V}\a_A \,\mu_E[\ent_{EA} f]\geq \g_E(\a)\,\ent_E f,
%&\sum_{A\subset V}\a_A \,\mu_O[\ent_{OA} f]\geq \g_O(\a)\,\ent_O f,
\end{align}
where $\g_E(\a)=\min_{x\in E}\sum_{A: A\ni x}\a_A$.
% for any choice of weights $\a=\{\a_A, A\subset V\}$. 
Similarly, 
\begin{align}\label{shear20}
%&\sum_{A\subset V}\a_A \,\mu_E[\ent_{EA} f]\geq \g_E(\a)\,\ent_E f,\\
&\sum_{A\subset V}\a_A \,\mu_O[\ent_{OA} f]\geq \g_O(\a)\,\ent_O f,
\end{align}
with $\g_O(\a)=\min_{x\in O}\sum_{A: A\ni x}\a_A$. Since $\g_E(\a)$ and $\g_O(\a)$  are both at least as large as $\g(\a)$,
the inequality \eqref{eofacto1} follows by summing \eqref{shear10} and \eqref{shear20}, %the two inequalities above, 
taking the expectation with respect to $\mu$ and noting that both $\mu[\ent_{EA} f]$ and $\mu[\ent_{OA} f]$ are at most $\mu[\ent_{A} f]$.
\end{proof}

The rest of this section is concerned with the proof of the factorization into even and odd blocks. Namely, we prove the following theorem, which together with Proposition \ref{prop:reduction} establishes the main result Theorem \ref{th:mainresult}.

\begin{theorem}\label{th:main}
Suppose that the spin system satisfies $SM(K,a)$ for some constants $K,a>0$. Then there exists a constant $C>0$ such that for all $V\in\bbF$, $\t\in \O_{V^c}$, for all $f\geq 0$,
 \begin{equation}\label{EOth1}
\ent_V^\t f \leq C\mu_V^\t\left[\ent_E f + \ent_O f\right].
 \end{equation}
 If instead the spin system satisfies $SM_L(K,a)$ for some constants $K,a>0$, $L\in\bbN$, then the same conclusion \eqref{EOth1} holds, provided we require that $V\in\bbF^{(L)}$. 
\end{theorem}

\subsection{Proof of Theorem \ref{th:main}}
The overall idea is to follow a recursive strategy based on a geometric construction introduced in \cite{BertiniCancriniCesi}, see also \cite{Cesi}. However, contrary to the problems studied in \cite{BertiniCancriniCesi,Cesi}, the error terms produced at each step of the iteration are too large in our setting to obtain directly the desired conclusion, see Theorem \ref{th:maind}, and we will need an additional recursive argument to finish the proof, see Theorem \ref{th:toy}.   We first carry out the proof under the spatial mixing assumption $SM(K,a)$, and then, in the end, consider the relaxed assumption $SM_L(K,a)$. 

\begin{definition}\label{def:dk}
Set $\ell_k=(3/2)^{k/d}$ and let $\bbF_k$ denote the set of all subsets $V\in \bbF$ such that, up to translation and permutation of the coordinates, $V$ is contained in the rectangle $$[0,\ell_{k+1}]\times\dots\times [0,\ell_{k+d}].$$
Let  $\d(k)$ denote the largest constant $\d>0$ such that 
 \begin{equation}\label{EOth10}
\d\,\ent_V^\t f \leq \mu_V^\t\left[\ent_E f + \ent_O f\right]
 \end{equation}
holds for all $V\in\bbF_k$, $\t\in\O_{V^c}$, and all $f:\O_V\mapsto\bbR_+$. 
\end{definition}
Note that $\d(k)\leq 1$ for any $k\in\bbN$ since  if e.g.\ $f=f(\si_E)$ is a function depending only on the spins at even sites then the right hand side in  \eqref{EOth10} is equal to $\mu_V^\t\left[\ent_E f\right]\leq \ent_V^\t f$. On the other hand, the next lemma guarantees that it is positive for all $k\in\bbN$. 
\begin{lemma}\label{lem:base}
For every $k\in\bbN$, $\d(k)>0$. 
\end{lemma}
\begin{proof}
If the spin system has no hard constraints one can use 
a perturbation argument from \cite{HS}, see e.g.\ \cite[Lemma 2.2]{CMT} for the application to our setting. In particular, one obtains that 
there exists a constant $C>0$ such that for all $k\in\bbN$: $$\d(k)\geq \exp{\left(-\,C\ell_k^d\right)}\,.$$

In the presence of hard constraints, in the case of irreducible permissive systems 
%if the spin system is irreducible 
one can argue as follows. %For any $f\geq 0$, 
It is known that any probability measure $\mu$ satisfies 
 \begin{equation}\label{rough1}
\ent f\leq C_0\log (1/\mu_*)\var \left(\sqrt f\right),
 \end{equation}
 with $\mu_* = \min_\si\mu(\si)$, where the minimum is restricted to $\si$ such that $\mu(\si)>0$, and $C_0$ is an absolute constant, see \cite[Corollary A.4]{DS}. Here $\var$ denotes the variance functional of $\mu$. For a finite permissive system in a region $V$ one has $\mu_*\geq e^{-C|V|}$ for some $C>0$ independent of $V$. Moreover, using the irreducibility assumption, a crude coupling argument shows that the spectral gap of the even/odd Markov chain is bounded away from zero in any fixed region $V\in\bbF$, see  \cite[Lemma 5.1]{BCSV}. In other words,  for some constant $C_1=C_1(k)$ one has
  \begin{equation}\label{rough2}
\var \left(g\right)\leq C_1\,\mu\left[\var_E\left(g\right) +\var_O\left(g\right)\right],
 \end{equation}
 for any function $g$. 
 Taking $g=\sqrt f$, the desired conclusion now follows from \eqref{rough1} and \eqref{rough2} using, for both $\mu_E$ and $\mu_O$, the well known inequality $\var(\sqrt f)\leq \ent f$, which holds for any probability measure, see e.g.\ \cite[Lemma 1]{LO}.

\end{proof}

Lemma \ref{lem:base} will be used as the base case for our induction.
%Theorem \ref{th:main} is a consequence of the following results. 
\begin{theorem}\label{th:maind}
Assume $SM(K,a)$. There exists a constant $k_0\in\bbN$ depending on $K,a,d$ such that 
  \begin{equation}\label{EOth2}
\d(k)\geq \left(1- \frac{10}{\ell_k\d(k-1)}\right)\d(k-1),\qquad k\geq k_0.
 \end{equation}
\end{theorem}
Theorem \ref{th:maind} can
only be useful if we know that $\d(k)$ is much larger than $1/\ell_k$ for $k$ large enough,
%still compatible with e.g.\ some $O(1/\ell_k)$ decay of $\d(k)$, 
and thus it is not sufficient to prove Theorem \ref{th:main}. The next result allows us to have an independent control on $\d(k)$ which, together with Theorem \ref{th:maind} implies the desired uniform bound of Theorem \ref{th:main}.
\begin{theorem}\label{th:toy}
Assume $SM(K,a)$. For any $\e>0$, there exists a constant $k_0\in\bbN$ depending on $K,a,d,\e$, such that 
\begin{equation}\label{toy1}
\d(k)\geq \ell_k^{-\e},\qquad k\geq k_0.
 \end{equation}
\end{theorem}
Theorem \ref{th:maind} and Theorem \ref{th:toy} are more than sufficient for our purpose. Indeed, using \eqref{toy1} and \eqref{EOth2}, taking for instance $\e=1/2$, we see that $$\ell_k\d(k-1)\geq \ell_k^{1/2}=(3/2)^{k/2d}\geq 10(6/5)^{k/d}$$ if $k$ is large enough, and therefore %one obtains 
 \begin{equation}\label{final_rec}
\d(k)\geq \left(1- (5/6)^{k/d}\right)\d(k-1)\geq \d(k_0)\prod_{j=k_0}^\infty (1-(5/6)^{j/d}).
 \end{equation}
Lemma \ref{lem:base} and \eqref{final_rec} imply $\inf_{k\in\bbN} \d(k) >0$, which concludes the proof of Theorem \ref{th:main} under the assumption $SM(K,a)$.

\subsection{Proof of Theorem \ref{th:maind}}
We start with a simple decomposition that will be used in the inductive step. Recall that $E=EV$ and $O=OV$ are the even and odd sites respectively, in the given region $V$.
\begin{lemma}\label{EOlemma1}
For any $A,B\in\bbF$ such that $V=A\cup B$, for any $f\geq 0$:
\begin{align}\label{EA1}
&\ent_E f=  \mu_E[\ent_{EA}f + \ent_{EB}\mu_{EA}f],\\&
\ent_O f=  \mu_O[\ent_{OA}f + \ent_{OB}\mu_{OA}f].
\end{align}
\end{lemma}
\begin{proof}
The decomposition in Lemma \ref{lem:telescope} shows that 
$$
\ent_E f=  \mu_E[\ent_{EA}f] + \ent_{E}\mu_{EA}f.
$$
Another application of that decomposition shows that
$$
\ent_{E}\mu_{EA}f=  \mu_E[\ent_{EB}\mu_{EA}f] + \ent_{E}\mu_{EB}\mu_{EA}f.
$$
However, the product property of $\mu_E$ implies that $\mu_{EB}\mu_{EA}f=\mu_E f$, and therefore 
$$
\ent_{E}\mu_{EB}\mu_{EA}f=0.
$$
The same argument applies to the case of odd sites. 
\end{proof}
%it is possible to show that 
Let us give a sketch of the main steps of the proof before entering the details. 
Suppose that 
$V=A\cup B\in\bbF_k$, and suppose that the assumption of Lemma \ref{lem:entBA} is satisfied.  
Then 
\begin{align}\label{EA2}
\ent f\leq   \mu\left[\ent_{A}f + \ent_{B}\mu_A f\right] + \theta(\e)\,\ent f
%\ent_{OA}f + \ent_{OB}\mu_A f]
%+ c(\e)\ent  f,
\end{align}
where we use the fact that $\ent\mu_A f\leq \ent f$. 
Now suppose furthermore that $A,B\in \bbF_{k-1}$. By definition of $\d(k)$ we then have 
\begin{align}\label{EAs2}
\d(k-1)\mu[\ent_A f]&\leq  \mu[\ent_{EA}f +\ent_{OA}f ]\,,\\
\d(k-1)\mu[\ent_B(\mu_Af)]&\leq \mu[\ent_{EB}\mu_A f + \ent_{OB}\mu_A f].
\end{align}
%
% and 
%By definition of $\g(V)$ one has
%\begin{gather}\label{EAs2}
%\mu[\ent_A f]\leq\g(A)\,  \mu[\ent_{EA}f +\ent_{OA}f ]\,,\\
%\mu[\ent_B(\mu_Af)]\leq\g(B)\, \mu[\ent_{EB}\mu_A f + \ent_{OB}\mu_A f],
%\end{gather}
%Assuming \eqref{Cesia1}, from \eqref{Cesi2} %using the factorization estimate on $A,B$ separately, 
%we obtain
%\begin{align}\label{EA2}
%\ent f\leq\g\,  \mu[\ent_{EA}f + \ent_{EB}\mu_A f + \ent_{OA}f + \ent_{OB}\mu_A f]
%+ c(\e)\ent  f,
%\end{align}
%where $\g=\max\{\g(A),\g(B)\}$. 
Therefore, using Lemma \ref{EOlemma1},
\begin{align}\label{EA3}
\d(k-1)\ent f&\leq \mu[\ent_{E}f +  \ent_{O}f] + \theta(\e)\d(k-1)\ent f\, + \\
& \quad +  \mu[\ent_{EB}\mu_A f -\ent_{EB}\mu_{EA} f +
\ent_{OB}\mu_A f - \ent_{OB}\mu_{OA} f]. \nonumber 
\end{align}
Disregarding the second line in \eqref{EA3} would allow us to obtain a bound of the form $$\d(k)\geq (1-\theta(\e))\d(k-1),$$ provided that an arbitrary  set $V\in\bbF_k$ can be decomposed into sets $A,B\in\bbF_{k-1}$ as above. 
We remark that if $\mu$ were a product over $A,B$ then by convexity
one would have  \begin{align}\label{EA4}
\mu[\ent_{EB}\mu_{A} f ]\leq \mu[\ent_{EB}\mu_{EA} f ],
\end{align}
and the same bound for odd sites. 
%Indeed, if $\mu=\mu_A\otimes\mu_B$ then by convexity:
%$$
%\mu_{EB}[\ent_{EB}\mu_{A} f ]=\mu_{EB}[\ent_{EB}\mu_{A}\mu_{EA} f ]\leq \mu_A\mu_{EB}[\ent_{EB}\mu_{EA} f ],
%$$
%and \eqref{EA4} follows by taking expectation with respect to $\mu$.  
Thus in the product case the second line in \eqref{EA3} may be neglected and we recover a factorization statement which is contained already in Lemma \ref{lem:parislemma}. 
In the case we are interested in however one has $A\cap B\neq \emptyset$ and we cannot hope for a bound like \eqref{EA4}. For an illustration of the problem, consider for instance the 1D case, with $V=\{1,\dots,n\}$, $A=\{1,\dots,m\}$ and $B=\{m-\ell,\dots,n\}$ for some integers $0< \ell <m<n$. 
Suppose that $m+1$ is even, and suppose that $f$ only depends on $\si_m$. 
Then, once all odd sites have been frozen, $\mu_{EA}f$ is a constant, and therefore
$\ent_{EB}\mu_{EA} f=0$. 
On the other hand, $\mu_{A}f$ depends on $\si_{m+1}$, since the conditional expectation $\mu_A$ depends non-trivially on $\si_{m+1}$, and thus we may well have $\ent_{EB}\mu_{A} f\neq 0$. 

Therefore, the second line of \eqref{EA3} does produce a nontrivial error term. 
At this point a fruitful idea from \cite{MarStFlour} comes to our rescue. Namely, one can average over many possible choices of the decomposition $V=A\cup B$ and hope that the averaging lowers the size of the overall error.  This strategy works very well if the error terms have an additive structure, such as in the case of \cite{Cesi}. Here there is no simple additive structure to exploit, and we resort to 
%except for the one expressing the entropy as a sum of entropies as in Lemma \ref{lem:telescope}. 
%Therefore, our argument will 
using the martingale-type decompositions from Lemma \ref{lem:telescope} to control the average error term by means of the global entropy $\ent f$, see Lemma \ref{lem:anyD}. This will be sufficient  to obtain  the recursive estimate \eqref{EOth2}. 
To implement this argument, we use a slightly different averaging procedure than in \cite{Cesi}.

We turn to the actual proof. We start with some geometric considerations, see Figure \ref{fig1} for a two-dimensional representation. Set $r:=\lfloor \tfrac16\,\ell_{k+d}\rfloor$, and define the rectangular sets 
 \begin{align}\label{geo1}
& Q:=[0,\ell_{k+1}]\times\dots\times [0,\ell_{k+d-1}]\times [\tfrac13\,\ell_{k+d},\ell_{k+d}]\,\\
& R_i:=[0,\ell_{k+1}]\times\dots\times [0,\ell_{k+d-1}]\times [0,\tfrac12\,\ell_{k+d}+i]\,,\quad i=0,\dots,r+1.
\end{align}
Suppose %$V\in \bbF_k\setminus\bbF_{k-1}$. 
%We may assume 
that $V\subset [0,\ell_{k+1}]\times\dots\times [0,\ell_{k+d}]$,
and define, for $i=1,\dots,r+1$:
\begin{align}\label{geo2}
B:=Q\cap V\,,\quad \text{and} \quad A_i:=\begin{cases}
(R_i\cap E)\cup (R_{i-1}\cap O) & \text{ if $i$ is even}\\
(R_i\cap O)\cup (R_{i-1}\cap E) & \text{ if $i$ is odd}
\end{cases}
\end{align}
where, as usual $E=EV$ and $O=OV$ denote the even and the odd sites of $V$ respectively. Define also 
\begin{align}\label{geo3}
\G_i=A_i\setminus A_{i-1}\,,\qquad i=2,\dots,r+1.
\end{align}
\begin{figure}[htp]
\center
\begin{subfigure}

\begin{tikzpicture}[scale=2]
\draw  [fill=gray, fill opacity=0.2] (1/4,1/4) -- (1/4,9/4) -- (3/4,9/4)--(3/4,10/4)--(7/4,10/4)--(7/4,4/4)--(5/4,4/4)--(5/4,1/4)--(1/4,1/4);
\draw [fill,opacity=0.2] (0,0) -- (0,10/4);
\draw [fill,opacity=0.2] (1/4,0) -- (1/4,10/4);
\draw [fill,opacity=0.2] (1/2,0) -- (1/2,10/4);
\draw [fill,opacity=0.2] (3/4,0) -- (3/4,10/4);
\draw [fill,opacity=0.2] (1,0) -- (1,10/4);
\draw [fill,opacity=0.2] (5/4,0) -- (5/4,10/4);
\draw [fill,opacity=0.2] (7/4,0) -- (7/4,10/4);
\draw [fill,opacity=0.2] (3/2,0) -- (3/2,10/4);
\draw [fill,opacity=0.2] (2,0) -- (2,10/4);
\draw [fill,opacity=0.2] (0,0) -- (2,0);
\draw [fill,opacity=0.2] (0,1/2) -- (2,1/2);
\draw [fill,opacity=0.2] (0,1/4) -- (2,1/4);
\draw [fill,opacity=0.2] (0,3/4) -- (2,3/4);
\draw [fill,opacity=0.2] (0,5/4) -- (2,5/4);
\draw [fill,opacity=0.2] (0,7/4) -- (2,7/4);
\draw [fill,opacity=0.2] (0,9/4) -- (2,9/4);
\draw [fill,opacity=0.2] (0,10/4) -- (2,10/4);
%\draw [fill,opacity=0.2] (0,11/4) -- (2,11/4);
\draw [fill,opacity=0.2] (0,1) -- (2,1);
\draw [fill,opacity=0.2] (0,3/2) -- (2,3/2);
\draw [fill,opacity=0.2] (0,2) -- (2,2);
\draw[fill] (0,0) circle   [radius=0.03];
\draw[fill] (1/4,0) circle [radius=0.03];
\draw[fill] (2/4,0) circle [radius=0.03];
\draw[fill] (3/4,0) circle [radius=0.03];
\draw[fill](4/4,0) circle [radius=0.03];
\draw[fill] (5/4,0) circle [radius=0.03];
\draw[fill] (6/4,0) circle [radius=0.03];
\draw[fill] (7/4,0) circle [radius=0.03];
\draw[fill] (8/4,0) circle [radius=0.03];
\draw[fill] (0,1/4) circle [radius=0.03];
\draw[fill] (1/4,1/4) circle [radius=0.03];
\draw[fill] (2/4,1/4) circle [radius=0.03];
\draw[fill] (3/4,1/4) circle [radius=0.03];
\draw[fill] (4/4,1/4) circle [radius=0.03];
\draw[fill] (5/4,1/4) circle [radius=0.03];
\draw[fill] (6/4,1/4) circle [radius=0.03];
\draw[fill] (7/4,1/4) circle [radius=0.03];
\draw[fill] (8/4,1/4) circle [radius=0.03];
\draw[fill] (0,1/2) circle [radius=0.03];
\draw[fill] (1/4,1/2) circle [radius=0.03];
\draw[fill] (2/4,1/2) circle [radius=0.03];
\draw[fill] (3/4,1/2) circle [radius=0.03];
\draw[fill] (4/4,1/2) circle [radius=0.03];
\draw[fill] (5/4,1/2) circle [radius=0.03];
\draw[fill] (6/4,1/2) circle [radius=0.03];
\draw[fill] (7/4,1/2) circle [radius=0.03];
\draw[fill] (8/4,1/2) circle [radius=0.03];
\draw[fill] (0,3/4) circle [radius=0.03];
\draw[fill=green] (1/4,3/4) circle [radius=0.03];
\draw[fill=green] (2/4,3/4) circle [radius=0.03];
\draw[fill=green] (3/4,3/4) circle [radius=0.03];
\draw[fill=green] (4/4,3/4) circle [radius=0.03];
\draw[fill=green] (5/4,3/4) circle [radius=0.03];
\draw[fill] (6/4,3/4) circle [radius=0.03];
\draw[fill] (7/4,3/4) circle [radius=0.03];
\draw[fill] (8/4,3/4) circle [radius=0.03];
\draw[fill] (0,1) circle [radius=0.03];
\draw[fill=green] (1/4,1) circle [radius=0.03];
\draw[fill=green] (2/4,1) circle [radius=0.03];
\draw[fill=green] (3/4,1) circle [radius=0.03];
\draw[fill=green] (4/4,1) circle [radius=0.03];
\draw[fill=green] (5/4,1) circle [radius=0.03];
\draw[fill=green] (6/4,1) circle [radius=0.03];
\draw[fill=green] (7/4,1) circle [radius=0.03];
\draw[fill] (8/4,1) circle [radius=0.03];
\draw[fill] (0,5/4) circle [radius=0.03];
\draw[fill=green] (1/4,5/4) circle [radius=0.03];
\draw[fill=green] (2/4,5/4) circle [radius=0.03];
\draw[fill=green] (3/4,5/4) circle [radius=0.03];
\draw[fill=green] (4/4,5/4) circle [radius=0.03];
\draw[fill=green] (5/4,5/4) circle [radius=0.03];
\draw[fill=green] (6/4,5/4) circle [radius=0.03];
\draw[fill=green] (7/4,5/4) circle [radius=0.03];
\draw[fill=black] (8/4,5/4) circle [radius=0.03];
\draw[fill] (0,6/4) circle [radius=0.03];
\draw[fill=green] (1/4,6/4) circle [radius=0.03];
\draw[fill=green] (2/4,6/4) circle [radius=0.03];
\draw[fill=green] (3/4,6/4) circle [radius=0.03];
\draw[fill=green] (4/4,6/4) circle [radius=0.03];
\draw[fill=green] (5/4,6/4) circle [radius=0.03];
\draw[fill=green] (6/4,6/4) circle [radius=0.03];
\draw[fill=green] (7/4,6/4) circle [radius=0.03];
\draw[fill] (8/4,6/4) circle [radius=0.03];
\draw[fill] (0,7/4) circle [radius=0.03];
\draw[fill=green] (1/4,7/4) circle [radius=0.03];
\draw[fill=green] (2/4,7/4) circle [radius=0.03];
\draw[fill=green] (3/4,7/4) circle [radius=0.03];
\draw[fill=green] (4/4,7/4) circle [radius=0.03];
\draw[fill=green] (5/4,7/4) circle [radius=0.03];
\draw[fill=green] (6/4,7/4) circle [radius=0.03];
\draw[fill=green] (7/4,7/4) circle [radius=0.03];
\draw[fill=black] (8/4,7/4) circle [radius=0.03]; 
\draw[fill=black] (0,8/4) circle [radius=0.03];
\draw[fill=green] (1/4,8/4) circle [radius=0.03];
\draw[fill=green] (2/4,8/4) circle [radius=0.03];
\draw[fill=green] (3/4,8/4) circle [radius=0.03];
\draw[fill=green] (4/4,8/4) circle [radius=0.03];
\draw[fill=green] (5/4,8/4) circle [radius=0.03];
\draw[fill=green] (6/4,8/4) circle [radius=0.03];
\draw[fill=green] (7/4,8/4) circle [radius=0.03];
\draw[fill] (8/4,8/4) circle [radius=0.03];
\draw[fill] (0,9/4) circle [radius=0.03];
\draw[fill=green] (1/4,9/4) circle [radius=0.03];
\draw[fill=green] (2/4,9/4) circle [radius=0.03];
\draw[fill=green] (3/4,9/4) circle [radius=0.03];
\draw[fill=green] (4/4,9/4) circle [radius=0.03];
\draw[fill=green] (5/4,9/4) circle [radius=0.03];
\draw[fill=green] (6/4,9/4) circle [radius=0.03];
\draw[fill=green] (7/4,9/4) circle [radius=0.03];
\draw[fill] (8/4,9/4) circle [radius=0.03];
\draw[fill] (0,10/4) circle [radius=0.03];
\draw[fill] (1/4,10/4) circle [radius=0.03];
\draw[fill] (2/4,10/4) circle [radius=0.03];
\draw[fill=green] (3/4,10/4) circle [radius=0.03];
\draw[fill=green] (4/4,10/4) circle [radius=0.03];
\draw[fill=green] (5/4,10/4) circle [radius=0.03];
\draw[fill=green] (6/4,10/4) circle [radius=0.03];
\draw[fill=green] (7/4,10/4) circle [radius=0.03];
\draw[fill] (8/4,10/4) circle [radius=0.03];
%\draw[fill] (0,11/4) circle [radius=0.03];
%\draw[fill] (1/4,11/4) circle [radius=0.03];
%\draw[fill] (2/4,11/4) circle [radius=0.03];
%\draw[fill] (3/4,11/4) circle [radius=0.03];
%\draw[fill] (4/4,11/4) circle [radius=0.03];
%\draw[fill] (5/4,11/4) circle [radius=0.03];
%\draw[fill] (6/4,11/4) circle [radius=0.03];
%\draw[fill] (7/4,11/4) circle [radius=0.03];
%\draw[fill] (8/4,11/4) circle [radius=0.03];
\node [below left] at  (0,0) {0}; 
\node [below right] at  (2,0) {$\ell_{k+1}$};
\node [left] at  (0,10.5/4) {$\ell_{k+2}$};
\node [left] at  (0,3/4) {$\frac{1}{3}\ell_{k+2}$};
\node [left] at  (0,5/4) {$\frac{1}{2}\ell_{k+2}$};
\end{tikzpicture}
\end{subfigure}
\hfill
\begin{subfigure}

\begin{tikzpicture}[scale=2]
\draw  [fill=gray, fill opacity=0.2] (1/4,1/4) -- (1/4,9/4) -- (3/4,9/4)--(3/4,10/4)--(7/4,10/4)--(7/4,4/4)--(5/4,4/4)--(5/4,1/4)--(1/4,1/4);
\draw [fill,opacity=0.2] (0,0) -- (0,10/4);
\draw [fill,opacity=0.2] (1/4,0) -- (1/4,10/4);
\draw [fill,opacity=0.2] (1/2,0) -- (1/2,10/4);
\draw [fill,opacity=0.2] (3/4,0) -- (3/4,10/4);
\draw [fill,opacity=0.2] (1,0) -- (1,10/4);
\draw [fill,opacity=0.2] (5/4,0) -- (5/4,10/4);
\draw [fill,opacity=0.2] (7/4,0) -- (7/4,10/4);
\draw [fill,opacity=0.2] (3/2,0) -- (3/2,10/4);
\draw [fill,opacity=0.2] (2,0) -- (2,10/4);
\draw [fill,opacity=0.2] (0,0) -- (2,0);
\draw [fill,opacity=0.2] (0,1/2) -- (2,1/2);
\draw [fill,opacity=0.2] (0,1/4) -- (2,1/4);
\draw [fill,opacity=0.2] (0,3/4) -- (2,3/4);
\draw [fill,opacity=0.2] (0,5/4) -- (2,5/4);
\draw [fill,opacity=0.2] (0,7/4) -- (2,7/4);
\draw [fill,opacity=0.2] (0,9/4) -- (2,9/4);
\draw [fill,opacity=0.2] (0,10/4) -- (2,10/4);
%\draw [fill,opacity=0.2] (0,11/4) -- (2,11/4);
\draw [fill,opacity=0.2] (0,1) -- (2,1);
\draw [fill,opacity=0.2] (0,3/2) -- (2,3/2);
\draw [fill,opacity=0.2] (0,2) -- (2,2);
\draw[fill] (0,0) circle   [radius=0.03];
\draw[fill] (1/4,0) circle [radius=0.03];
\draw[fill] (2/4,0) circle [radius=0.03];
\draw[fill] (3/4,0) circle [radius=0.03];
\draw[fill](4/4,0) circle [radius=0.03];
\draw[fill] (5/4,0) circle [radius=0.03];
\draw[fill] (6/4,0) circle [radius=0.03];
\draw[fill] (7/4,0) circle [radius=0.03];
\draw[fill] (8/4,0) circle [radius=0.03];
\draw[fill] (0,1/4) circle [radius=0.03];
\draw[fill=yellow] (1/4,1/4) circle [radius=0.03];
\draw[fill=yellow] (2/4,1/4) circle [radius=0.03];
\draw[fill=yellow] (3/4,1/4) circle [radius=0.03];
\draw[fill=yellow] (4/4,1/4) circle [radius=0.03];
\draw[fill=yellow] (5/4,1/4) circle [radius=0.03];
\draw[fill] (6/4,1/4) circle [radius=0.03];
\draw[fill] (7/4,1/4) circle [radius=0.03];
\draw[fill] (8/4,1/4) circle [radius=0.03];
\draw[fill] (0,1/2) circle [radius=0.03];
\draw[fill=yellow] (1/4,1/2) circle [radius=0.03];
\draw[fill=yellow] (2/4,1/2) circle [radius=0.03];
\draw[fill=yellow] (3/4,1/2) circle [radius=0.03];
\draw[fill=yellow] (4/4,1/2) circle [radius=0.03];
\draw[fill=yellow] (5/4,1/2) circle [radius=0.03];
\draw[fill] (6/4,1/2) circle [radius=0.03];
\draw[fill] (7/4,1/2) circle [radius=0.03];
\draw[fill] (8/4,1/2) circle [radius=0.03];
\draw[fill] (0,3/4) circle [radius=0.03];
\draw[fill=yellow] (1/4,3/4) circle [radius=0.03];
\draw[fill=yellow] (2/4,3/4) circle [radius=0.03];
\draw[fill=yellow] (3/4,3/4) circle [radius=0.03];
\draw[fill=yellow] (4/4,3/4) circle [radius=0.03];
\draw[fill=yellow] (5/4,3/4) circle [radius=0.03];
\draw[fill] (6/4,3/4) circle [radius=0.03];
\draw[fill] (7/4,3/4) circle [radius=0.03];
\draw[fill] (8/4,3/4) circle [radius=0.03];
\draw[fill] (0,1) circle [radius=0.03];
\draw[fill=yellow] (1/4,1) circle [radius=0.03];
\draw[fill=yellow] (2/4,1) circle [radius=0.03];
\draw[fill=yellow] (3/4,1) circle [radius=0.03];
\draw[fill=yellow] (4/4,1) circle [radius=0.03];
\draw[fill=yellow] (5/4,1) circle [radius=0.03];
\draw[fill=yellow] (6/4,1) circle [radius=0.03];
\draw[fill=yellow] (7/4,1) circle [radius=0.03];
\draw[fill] (8/4,1) circle [radius=0.03];
\draw[fill] (0,5/4) circle [radius=0.03];
\draw[fill=yellow] (1/4,5/4) circle [radius=0.03];
\draw[fill=yellow] (2/4,5/4) circle [radius=0.03];
\draw[fill=yellow] (3/4,5/4) circle [radius=0.03];
\draw[fill=yellow] (4/4,5/4) circle [radius=0.03];
\draw[fill=yellow] (5/4,5/4) circle [radius=0.03];
\draw[fill=yellow] (6/4,5/4) circle [radius=0.03];
\draw[fill=yellow] (7/4,5/4) circle [radius=0.03];
\draw[fill] (8/4,5/4) circle [radius=0.03];
\draw[fill] (0,6/4) circle [radius=0.03];
\draw[fill=red] (1/4,6/4) circle [radius=0.03];
\draw[fill=yellow] (2/4,6/4) circle [radius=0.03];
\draw[fill=red] (3/4,6/4) circle [radius=0.03];
\draw[fill=yellow] (4/4,6/4) circle [radius=0.03];
\draw[fill=red] (5/4,6/4) circle [radius=0.03];
\draw[fill=yellow] (6/4,6/4) circle [radius=0.03];
\draw[fill=red] (7/4,6/4) circle [radius=0.03];
\draw[fill] (8/4,6/4) circle [radius=0.03];
\draw[fill] (0,7/4) circle [radius=0.03];
\draw[fill] (1/4,7/4) circle [radius=0.03];
\draw[fill=red] (2/4,7/4) circle [radius=0.03];
\draw[fill] (3/4,7/4) circle [radius=0.03];
\draw[fill=red] (4/4,7/4) circle [radius=0.03];
\draw[fill] (5/4,7/4) circle [radius=0.03];
\draw[fill=red] (6/4,7/4) circle [radius=0.03];
\draw[fill] (7/4,7/4) circle [radius=0.03];
\draw[fill] (8/4,7/4) circle [radius=0.03]; 
\draw[fill] (0,8/4) circle [radius=0.03];
\draw[fill] (1/4,8/4) circle [radius=0.03];
\draw[fill] (2/4,8/4) circle [radius=0.03];
\draw[fill] (3/4,8/4) circle [radius=0.03];
\draw[fill] (4/4,8/4) circle [radius=0.03];
\draw[fill] (5/4,8/4) circle [radius=0.03];
\draw[fill] (6/4,8/4) circle [radius=0.03];
\draw[fill] (7/4,8/4) circle [radius=0.03];
\draw[fill] (8/4,8/4) circle [radius=0.03];
\draw[fill] (0,9/4) circle [radius=0.03];
\draw[fill] (1/4,9/4) circle [radius=0.03];
\draw[fill] (2/4,9/4) circle [radius=0.03];
\draw[fill] (3/4,9/4) circle [radius=0.03];
\draw[fill] (4/4,9/4) circle [radius=0.03];
\draw[fill] (5/4,9/4) circle [radius=0.03];
\draw[fill] (6/4,9/4) circle [radius=0.03];
\draw[fill] (7/4,9/4) circle [radius=0.03];
\draw[fill] (8/4,9/4) circle [radius=0.03];
\draw[fill] (0,10/4) circle [radius=0.03];
\draw[fill] (1/4,10/4) circle [radius=0.03];
\draw[fill] (2/4,10/4) circle [radius=0.03];
\draw[fill] (3/4,10/4) circle [radius=0.03];
\draw[fill] (4/4,10/4) circle [radius=0.03];
\draw[fill] (5/4,10/4) circle [radius=0.03];
\draw[fill] (6/4,10/4) circle [radius=0.03];
\draw[fill] (7/4,10/4) circle [radius=0.03];
\draw[fill] (8/4,10/4) circle [radius=0.03];
%\draw[fill] (0,11/4) circle [radius=0.03];
%\draw[fill] (1/4,11/4) circle [radius=0.03];
%\draw[fill] (2/4,11/4) circle [radius=0.03];
%\draw[fill] (3/4,11/4) circle [radius=0.03];
%\draw[fill] (4/4,11/4) circle [radius=0.03];
%\draw[fill] (5/4,11/4) circle [radius=0.03];
%\draw[fill] (6/4,11/4) circle [radius=0.03];
%\draw[fill] (7/4,11/4) circle [radius=0.03];
%\draw[fill] (8/4,11/4) circle [radius=0.03];
\node [below left] at  (0,0) {0}; 
\node [below right] at  (2,0) {$\ell_{k+1}$};
\node [left] at  (0,10.5/4) {$\ell_{k+2}$};
\node [left] at  (0,3/4) {$\frac{1}{3}\ell_{k+2}$};
\node [left] at  (0,5/4) {$\frac{1}{2}\ell_{k+2}$};
\end{tikzpicture}
\end{subfigure}

\caption{\footnotesize{The gray area is the volume $V$. Left: $B$ is the set of green vertices. Right: $A_1$ is the set of  yellow vertices, $\G_2$ is the set of red vertices, and $A_2$ is the set of yellow and red vertices together.}}
\label{fig1}
\end{figure}
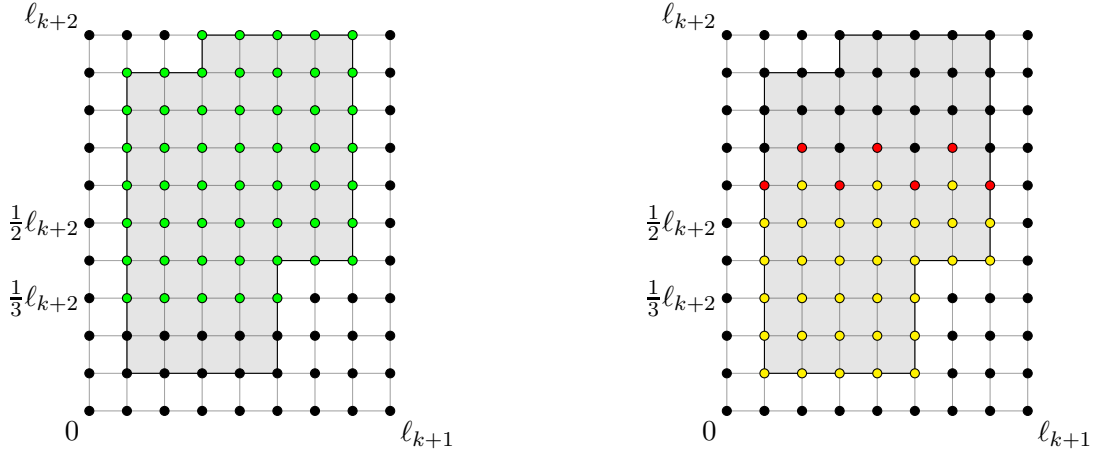

%We need the following simple observations. 
\begin{lemma}\label{lem:geo}
Suppose that $V\subset [0,\ell_{k+1}]\times\dots\times [0,\ell_{k+d}]$, and that $V\notin \bbF_{k-1}$. 
Referring to the above setting,  for all $i=1,\dots
r$: 
\begin{enumerate}
\item $V=A_i\cup B$, $V\setminus B\neq \emptyset$ and $V\setminus A_i\neq \emptyset$;
\item $d(V\setminus B,V\setminus A_i)\geq \frac14\,\ell_k$;
\item $B\in \bbF_{k-1}$ and $A_i\in \bbF_{k-1}$;
\item $\G_{i+1}\subset E$ if $i$ is odd, and $\G_{i+1}\subset O$ if $i$ is even. Moreover $A_{i}$ and $V\setminus A_{i+1}$ become independent if we condition on the spins in $\G_{i+1}$, that is 
\begin{align}\label{geo4}
\mu_V\left(\cdot|\si_{\,\G_{i+1}}\right)=\mu_{V\setminus \G_{i+1}}=\mu_{A_{i}}\mu_{V\setminus A_{i+1}}= 
\mu_{V\setminus A_{i+1}}\mu_{A_{i}}.
\end{align}
\end{enumerate} 
\end{lemma}
\begin{proof}
1. Suppose that $V\setminus B$ is empty. Then $V=B$ and therefore, up to translation it is contained in $[0,\ell_{k+1}]\times\dots\times [0,\tfrac23\ell_{k+d}]$. Since $\tfrac23\ell_{k+d}= \ell_k$ this would imply that up to   permutation of the coordinates $V\in [0,\ell_{k}]\times [0,\ell_{k+1}]\times\dots\times [0,\ell_{k+d-1}]$ which violates the assumption $V\notin \bbF_{k-1}$. The same argument shows that $R_{i-1}\cap V\neq \emptyset$ for all $i$ and $A_i\neq \emptyset$ follows from $A_i\supset R_{i-1}\cap V$.

\smallskip

\noindent
2. If $x\in V\setminus B$ and $y\in V\setminus A_i$ then $y_d-x_d\geq \frac12\ell_{k+d}-\frac13\ell_{k+d}=\frac16\ell_{k+d}=\frac14\ell_k$.

\smallskip

\noindent
3. The maximal stretch of $B$ along the $d$-th coordinate is at most $\frac23 \ell_{k+d}=\ell_k$ and therefore up to 
translations and permutation of the coordinates $B\in [0,\ell_{k}]\times [0,\ell_{k+1}]\times\dots\times [0,\ell_{k+d-1}]$ which says that $B\in \bbF_{k-1}$. The same argument shows that $A_i\subset R_i\cap V\in \bbF_{k-1}$ for all $i$.

\smallskip

\noindent
4. If $i\geq 1$ is odd, then
\begin{align*}
\G_{i+1}&= \left[(R_{i+1}\cap E)\cup (R_{i}\cap O)\right] \setminus 
\left[(R_{i}\cap O)\cup (R_{i-1}\cap E)\right] \\&= (R_{i+1}\cap E)\setminus (R_{i-1}\cap E),
\end{align*}
and therefore $\G_{i+1}\subset E$. Similarly, one has $\G_{i+1}\subset O$ if $i$ is even.
%Moreover, $\G_i\cap \G_j=\emptyset$ for all $i\neq j$.
Moreover,  any $\bbZ^d$-path inside $V$ connecting $A_{i}$ with $V\setminus A_{i+1}$ must go through $\G_{i+1}$, and therefore $A_{i}$ and $V\setminus A_{i+1}$ become independent if we condition on the spins in $\G_{i+1}$.
\end{proof}

\begin{lemma}\label{lem:ABlemmad}
Let $V$, $B$ and $A_i$ be as in Lemma \ref{lem:geo}. If $SM(K,a)$ holds, then 
\begin{equation}\label{Cesiab1}
\|\mu_B\mu_{A_i} g - \mu g\|_\infty \leq \e_k \mu (|g|)\,,\qquad \e_k=5^{d}K\ell_{k}^{d-1}e^{-a\ell_k/4}\,,
\end{equation}
for all $i=1,\dots r$, all functions $g\in L^1(\mu)$, and for all $k\geq k_0=k_0(K,a,d)$.
 \end{lemma}
\begin{proof}
Since $i$ is fixed, for simplicity we write $A$ instead of $A_i$.  Set $h=\mu_A g$. Then $h$ depends only on $\si_\D$, where $\D=V\setminus A\subset B$. We are going to use \eqref{ssmca} with $\L=B$. %For any fixed 
Let $\O_{B,\t}$ denote the set of all spin configurations $\eta\in\O_{B^c}$ which agree  on the set $V^c$ with the overall boundary condition $\t\in\O_{V^c}$. For any $\eta\in\O_{B,\t}$ one has
\begin{align}\label{ABlemma1}
\mu^\eta_B\left(\mu_{A}g\right)- \mu g &=  
 \int\mu^\t_{V,V\setminus B}(d\eta')
\left(\mu_{B,\D}^\eta h-\mu^{\eta'}_{B,\D} h\right) 
\\&
= \int\mu^\t_{V,V\setminus B}(d\eta')
\int\mu^{\eta'}_{B,\D}(d\si)
\left(\frac{\psi_{B,\D}^{\eta}(\si)}{\psi_{B,\D}^{\eta'}(\si)}-1\right)
h(\si).
\end{align}
Therefore,
\begin{align}\label{ABlemma2}
\|\mu_{B}\mu_A g - \mu g\|_\infty &\leq 
\e\,\mu(|h|)\leq \e\,\mu(|g|),
\end{align}
where 
\begin{align}\label{ABlemma3}
\e:=
\sup_{\eta,\eta'\in\O_{B,\t }}
\left\|\frac{\psi_{B,\D}^{\eta}}{\psi_{B,\D}^{\eta'}}-1\right\|_\infty.
\end{align}
Since $\psi_{B,\D}^{\eta}$ depends on $\eta$ only through the spins in $\partial B$, 
the configurations $\eta,\eta'\in\O_{B,\t }$ in \eqref{ABlemma3} can be assumed to differ only in the set $N_B=(\partial B)\cap (V\setminus B)$. Notice that $N_B$ has at most $(\ell_{k+d-1}+1)^{d-1}$ elements, and that $$d(N_B,\D)\geq 
d(V\setminus B,V\setminus A)
\geq \tfrac14\ell_k,$$ by Lemma \ref{lem:geo}(2). 
%In \eqref{ABlemma3} we may assume that $\t_{B^c}$ and $\si_{B^c}$ are equal at all sites $B^c\setminus N_B$. 
Therefore, if $\eta(0)=\eta,\dots,\eta(m)=\eta'$, denotes a sequence of configurations interpolating between $\eta$ and $\eta'$, such that, for all $j\in\{0,\dots,m-1\}$,  $\eta(j)$ and $\eta(j+1)$ differ only at one site $x_j\in N_B$, with $m\leq (\ell_{k+d-1}+1)^{d-1}$, we have
\begin{align}\label{ABlemma4}
\frac{\psi_{B,\D}^{\eta}}{\psi_{B,\D}^{\eta'}}
%\frac{\mu_{B,\D}^{\t_{B^c}}(\si_\D)}{\mu_{B,\D}^{\si_{B^c}}(\si_\D)} 
= \prod_{j=1}^m %\phi_j\,,\qquad \phi_j = 
\frac{\psi_{B,\D}^{\eta(j-1)}}{\psi_{B,\D}^{\eta(j)}}.
\end{align}
The definition of $SM(K,a)$ implies that 
\begin{align}\label{ABlemma14}
\left\|\frac{\psi_{B,\D}^{\eta(j-1)}}{\psi_{B,\D}^{\eta(j)}} -1 \right\|_\infty 
\leq \e_0:=Ke^{-a \ell_{k}/4}.
\end{align}
Expanding the products in \eqref{ABlemma3}, and assuming $m\e_0\leq 1$, we obtain
\begin{align}\label{ABlemma13}
\e\leq \sum_{\ell=1}^m\binom{m}{\ell}\e_0^\ell = (1+\e_0)^m-1\leq em\e_0,
\end{align}
where we use the inequality $(1+x)^m\leq 1+emx$ for $x>0$ and $m>0$ such that $mx\leq 1$. 
%interpolating in $N_A$ we obtain 
Thus, if $k\geq k_0$ for some constant $k_0$ depending only on $K,a,d$, we have obtained 
\eqref{ABlemma2} with $\e=K'\ell_{k}^{d-1}e^{-a \ell_{k}/4}$, where $K'=(3/2)^{d-1}eK\leq 5^dK$.
\end{proof}

\begin{lemma}\label{lem:anyD}
Let $V$, $B$ and $A_i$, $i=1,\dots r$, be as in Lemma \ref{lem:geo}. Then
\begin{align}\label{a1D1}
\sum_{i=1}^{r}\,
&\mu\left[\ent_{EB}\mu_{A_i} f -\ent_{EB}\mu_{EA_i} f \right]\leq \ent f\,,\\
\sum_{i=1}^{r}
\,&
\mu\left[\ent_{OB}\mu_{A_i} f -\ent_{OB}\mu_{OA_i} f \right]\leq \ent f.
\end{align}
%The same holds for odd sites, that is replacing $E$ with $O$ in the above expression. 
\end{lemma}
\begin{proof}
We prove the first inequality. The same argument proves the second one, with the role of even and odd sites exchanged.
Fix $i\in\{1,\dots,r\}$. Notice that $\mu_{A_i} f = \mu_{A_i}\mu_{EA_i}f$. Let us first observe that if $i$ is even then
\begin{align}\label{a1D2}
\mu\left[\ent_{EB}\mu_{A_i} f  -\ent_{EB}\mu_{EA_i} f \right]\leq 0.
\end{align}
Indeed, in this case $i+1$ is odd and Lemma \ref{lem:geo}(4) implies 
%the average $\mu_{A_i}$ does not depend on the even sites in $B\setminus A$, so that if all odd sites are frozen, then 
\begin{align}\label{a1D4}
\mu_{EB}\mu_{A_i} f = \mu_{EB}\mu_{A_i} \mu_{EA_i} f
=\mu_{A_i} \mu_{EB}\mu_{EA_i} f.
\end{align}
Therefore,
\begin{align}\label{a1D3}
\mu\left[\ent_{EB}\mu_{A_i} f \right]&= \mu\left[\mu_{A_i} f \log\left(\mu_{A_i} f/\mu_{EB}\mu_{A_i} f\right)\right]\\ &=
\mu\left[\mu_{A_i} \mu_{EA_i} f \log\left(\mu_{A_i} \mu_{EA_i} f/\mu_{A_i}\mu_{EB} \mu_{EA_i} f\right)\right]
\\ &=
\mu\left[\mu_{EA_i} f \log\left(\mu_{A_i} \mu_{EA_i} f/\mu_{A_i} \mu_{EB}\mu_{EA_i} f\right)\right]
\\ &=
\mu\left[\mu_{EB}\left(\mu_{EA_i} f \log\left(\mu_{A_i} \mu_{EA_i} f/\mu_{EB}\mu_{A_i} \mu_{EA_i} f\right)\right)\right]
\\ &\leq
\mu\left[\mu_{EB}\left(\mu_{EA_i} f \log\left( \mu_{EA_i} f/\mu_{EB} \mu_{EA_i} f\right)\right)\right]
\\ &= \mu\left[\ent_{EB}\mu_{EA_i} f \right],
\end{align}
where the inequality follows from the variational principle \eqref{varprin}. This settles the case when $i$ is even. 

Next, suppose that $i$ is odd. 
Here the commutation relation  \eqref{a1D4} does not hold, since the average $\mu_{A_i}$ depends on the spins in the even sites $\G_{i+1}\subset B\setminus A_i$. Moreover, \eqref{a1D2} is in general false  since  if e.g.\ $f$ depends only on $\si_{\,\G_i}$, then $\ent_{EB}\mu_{EA_i} f=0$ while one can have $\ent_{EB}\mu_{A_i} f>0$. 

Define $g=\mu_{EA_i}f$.  From the decomposition in Lemma \ref{lem:telescope} we see that
\begin{align}\label{a1D31a}
\ent_{EB} (\mu_{EA_i} f) & = \ent_{E} (g) \\& 
= \ent_{E} \left(\mu_E(g|\si_{i+1})\right) + \mu_{E}\left[\ent_E\left(g|\si_{i+1}\right)\right],
\end{align}
where we use the shorthand notation $\si_{i+1}$ for $\si_{\,\G_{i+1}}$,  $\ent_E\left(g|\si_{i+1}\right)$ denotes the entropy of $g$ with respect to the conditional measure $\mu_E(\cdot|\si_{i+1})=\mu_{E\setminus\G_{i+1}}$. 
%From Lemma \ref{lem:geo}(4) it follows  that, conditioned on all odd sites, the sets $\G_{i+1}, \G_{i+3},\dots$, become independent. Therefore
Since $\mu_E$ is a product measure, 
 \begin{align}\label{a1D31a2}
\ent_{E} \left(\mu_E(g|\si_{i+1})\right)
=\ent_{\, i+1}\left(\mu_E(g|\si_{i+1})\right),
\end{align}
where $\ent_{\, i+1}=\ent_{\G_{i+1}}$ denotes the entropy with respect to the probability measure $\mu_{\G_{i+1}}$.  Similarly,
\begin{align}\label{a1D301}
\ent_{EB} (\mu_{A_i} f) & = \ent_{E} (\mu_{A_i}g) \\& 
= \ent_{\, i+1} \left(\mu_E(\mu_{A_i} g|\si_{i+1})\right) + \mu_{E}\left[\ent_E\left(\mu_{A_i} g|\si_{i+1}\right)\right].
\end{align}
Let us show that 
\begin{align}\label{a1D302}
 \mu\left[\ent_E\left(\mu_{A_i} g|\si_{i+1}\right)\right]\leq
  \mu\left[\ent_E\left(g|\si_{i+1}\right)\right].
  \end{align}
Indeed, Lemma \ref{lem:geo}(4) implies that $$\mu_{E}(\mu_{A_i} g | \si_{i+1})  = \mu_{E(V\setminus A_{i+1})} \mu_{A_i}g=  \mu_{A_i}\mu_{E(V\setminus A_{i+1})} g= \mu_{A_i}\mu_{E}(g | \si_{i+1}),$$ 
where $E(V\setminus A_{i+1})$ are the even sites in $V\setminus A_{i+1}$, and we have used the fact that $A_i$ and $E(V\setminus A_{i+1})$ are conditionally independent given the spins $\si_{i+1}$.  
Therefore, reasoning as in \eqref{a1D3}: 
\begin{align}\label{a1D300}
\mu\left[\ent_E\left(\mu_{A_i} g|\si_{i+1}\right)\right]&= 
\mu\left[\mu_{A_i} g \log\left(\mu_{A_i} g/\mu_{E}(\mu_{A_i} g | \si_{i+1})\right)\right]
\\ &=
\mu\left[\mu_{A_i} g \log\left(\mu_{A_i} g/\mu_{A_i}\mu_{E}( g | \si_{i+1})\right)\right]
\\ &=
\mu\left[g \log\left(\mu_{A_i} g/\mu_{E}(\mu_{A_i} g | \si_{i+1})\right)\right]
\\ &\leq
\mu\left[g \log\left(g/\mu_{E}(g | \si_{i+1})\right)\right] \\& =  \mu\left[\ent_E\left(g|\si_{i+1}\right)\right].
\end{align}
From \eqref{a1D31a}-\eqref{a1D31a2}-\eqref{a1D301}-\eqref{a1D302}
we conclude that, when $i$ is odd:
\begin{align}\label{a1D36}
&\mu\left[\ent_{EB}\mu_{A_i} f -\ent_{EB}\mu_{EA_i} f \right]\\
&\qquad \leq 
\mu[\ent_{\, i+1} \left(\mu_E(\mu_{A_i} g|\si_{i+1})\right)]- \mu[\ent_{\, i+1} \left(\mu_E(g|\si_{i+1})\right)].
\end{align}
%
%
%recalling \eqref{1D2} when $i$ is even, we have obtained
% \begin{align}\label{1D37}
%\sum_{i=i_{\rm min}}^{i_{\rm max}}
%\mu\left[\ent_{EB}\mu_{A_i} f -\ent_{EB}\mu_{EA_i} f \right]
%\leq \sum_{i=i_{\rm min}}^{i_{\rm max}}\mu[\ent_{\, i+1} (\mu_{A_i}g_{i+1})].
%\end{align}
As in \eqref{a1D302}, we may write %recalling \eqref{a1D34} at $j=1$, we see that 
\begin{align*}
\mu_E(\mu_{A_i} g|\si_{i+1})
&= %\mu_E(\mu_{A_i} g| \si_{i+1}) = 
\mu_E(\mu_{A_i} f| \si_{i+1}) 
%=\mu_{A_i}\mu_E(f| \si_{i+1})\\&=\mu_{A_i\cup E(V\setminus A_{i+1})} f
= \mu_{E(V\setminus A_{i+1})} \mu_{A_i}f\,.
%= \mu_{E(V\setminus A_{i+1})}\left[\mu(f|\si_{i+1},\si_{i+2},\dots,\si_n)\right]\,,
\end{align*}
Therefore
\begin{align}\label{a1D38}
\mu[\ent_{\, i+1} \left(\mu_E(\mu_{A_i} g|\si_{i+1})\right)]
%\mu[\ent_{\, i+1} (\mu_{A_i}g_{i+1})]
&=
\mu\left[\ent_{\, i+1} \left(\mu_{E(V\setminus A_{i+1})} \mu_{A_i}f\right)\right]
%\mu\left[\ent_{\, i+1} \left(\mu_{E(V\setminus A_{i+1})}\left[\mu(f|\si_{i+1},\si_{i+2},\dots,\si_n)\right]\right)\right]
\\&
\leq\mu\left[\mu_{E(V\setminus A_{i+1})}\ent_{\, i+1} \mu_{A_i}f\right]\\&=
\mu\left[\ent_{\, i+1} \mu_{A_i}f\right],
\\&
\leq \mu\left[\ent_{A_{i+1}} \mu_{A_i}f\right],
%\mu\left[\ent_{\, i+1} (\mu(f| \si_{i+1},\si_{i+2},\dots,\si_n)\right],
\end{align}
where the first inequality follows from convexity of entropy and the second from the monotonicity of $A\mapsto \mu[\ent_A f]$. Neglecting the last term in \eqref{a1D36}, 
we have arrived at 
\begin{align}\label{a1D360}
\mu\left[\ent_{EB}\mu_{A_i} f -\ent_{EB}\mu_{EA_i} f \right]
 \leq  \mu\left[\ent_{A_{i+1}} \mu_{A_i}f\right],
\end{align}
for all $i$ odd. In view of the estimate \eqref{a1D2} we may use the bound \eqref{a1D360} for all $i$.
Therefore, an application of Lemma \ref{lem:telescope}
%the telescoping argument \eqref{paris4} 
shows that 
%for any $m\leq r$
%\begin{align}\label{a1D40}
%\sum_{i=0}^{m-1}\mu\left[\ent_{A_{i+1}} \mu_{A_i}f\right]\leq \ent f.
%%\mu[\ent_{\{1,\dots,i+1\}} (\mu(f| \si_{i+1},\si_{i+2},\dots,\si_n)]
%%=\mu\left[\ent_{A_m}(\mu_{A_0}f)\right].
%\end{align}
%%where $A_i=\cup_{j=1}^i\G_j$. %and $\mu_{A_0}f=f$. 
%Summarizing, neglecting the last term in \eqref{a1D36}, 
%we have arrived at 
 \begin{align}\label{a1D41}
\sum_{i=1}^{r}
\mu\left[\ent_{EB}\mu_{A_i} f -\ent_{EB}\mu_{EA_i} f \right]&
\leq  \sum_{i=1}^{r}\mu\left[\ent_{A_{i+1}} \mu_{A_i}f\right]\\&
= \mu\left[\ent_{A_{r+1}} \mu_{A_1}f\right]
%&\leq \mu\left[\ent_{A_{r}}(\mu_{A_0}f)\right]
%\\&
\leq \ent f. \end{align}
\end{proof}

We are now able to conclude the proof of Theorem \ref{th:maind}. 
To prove the recursive bound \eqref{EOth2} we suppose $V\in\bbF_k\setminus \bbF_{k-1}$. 
Then, by translation invariance and by the invariance under coordinate permutation, we may assume that $V$ is as in Lemma \ref{lem:geo}.
Combining Lemma \ref{lem:entBA} with Lemma \ref{lem:ABlemmad} we obtain, for each $i=1,\dots,
r$,
 \begin{align}\label{EOth20}
(1-\theta(\e_k))\,\ent f\leq \mu\left[\ent_{A_i}f + \ent_B \mu_{A_i}f\right].
\end{align}
Since $A_i,B\in\bbF_{k-1}$, by definition of $\d(k)$ we obtain
 \begin{align}\label{EOth21}
&(1-\theta(\e_k))\d(k-1)\,\ent f\\&\qquad \leq \mu\left[\ent_{EA_i}f + \ent_{EB} \mu_{A_i}f + \ent_{OA_i}f + \ent_{OB} \mu_{A_i}f\right].
\end{align}
From Lemma \ref{EOlemma1} we find that the right hand side of \eqref{EOth21} equals
\begin{align}\label{EOth22}
&\mu\left[\ent_{E}f + \ent_{O} f\right] + \\&\qquad + \mu\left[\ent_{EB}\mu_{A_i}f - \ent_{EB}\mu_{EA_i}f \right]
+   \mu\left[\ent_{OB}\mu_{A_i}f - \ent_{OB}\mu_{OA_i}f \right].
\end{align}
Averaging over $i$ in \eqref{EOth22} and using Lemma \ref{lem:anyD}, 
 \begin{align}\label{EOth210}
&(1-\theta(\e_k))\d(k-1)\,\ent f
%\\&\qquad 
\leq \mu\left[\ent_{E}f + \ent_{O}f\right] + \frac2r \,\ent f.
\end{align}
In conclusion, $\d(k)\geq (1-\theta(\e_k))\d(k-1) - \frac2r$, or equivalently $$\d(k)\geq \left(1-\theta(\e_k) - \frac2{r\d(k-1)}\right)\d(k-1).$$
Since $r\sim \frac14\ell_k$ and $\d(k-1)\leq 1$, it follows that  $\frac1{r\d(k-1)}\gg \theta(\e_k)$ for all $k$
large enough, and therefore $$\d(k)\geq \left(1- \frac{10}{\d(k-1)\ell_k}\right)\d(k-1),$$ for all $k\geq k_0(K,a,d)$.
%, which es the proof of Theorem \ref{th:maind}. 

\subsection{Proof of Theorem \ref{th:toy}}
Here we shall use again a recursion on an exponential scale. However, this time we divide the set $V$ into two sets $A=\cup_i A_i$, $B=\cup_i B_i$ each being the union of a large number of well separated subsets. 
%The result is then obtained by a combination of t
We use the factorization from Lemma \ref{lem:entBA} to reduce the problem in the set $V$ to the problem in either $A$ or $B$. Then we use the Lemma \ref{lem:parislemma} to tensorize within $A$ and within $B$, which allows us to reduce the problem to a single region $A_i$ or $B_i$ only.  

Fix a large integer $b>1$, define $u_k = b^{k/d}$, and call $\bbG_k$ the set of all subsets $V\subset \bbZ^d$ which up to translations and permutation of the coordinates are included in the rectangle $[0,u_{k+1}]\times\cdots\times[0,u_{k+d}]$.
We partition the interval $I=[0,u_{k+d}]$  into $2b$ consecutive non-overlapping intervals $I_1,\dots,I_{2b}$ such that $I_j$ have length $t_k:=\frac1{2b}u_{k+d}$, that is 
$$I_j=[(j-1)t_k,j t_k], \qquad j=1,\dots,2b.$$  
 Define also the enlarged intervals $\bar I_j = \{s\in I: d(s,I_j)\leq t_k/4\}$, and consider the collections of intervals
$$
\D_A= \bigcup_{j=1}^{2b} \bar I_{j}\,\ind_{j \text{ odd}}, \qquad \D_B= \bigcup_{j=1}^{2b} \bar I_{j}\,\ind_{j \text{ even}}. 
$$
We remark that both $\D_A$ and $\D_B $ are collections of non-overlapping intervals, with  $$d(\bar I_{2j-1},\bar I_{2i-1})\geq\frac12\, t_k\,,\qquad d(\bar I_{2j},\bar I_{2i})\geq\frac12\, t_k$$ for all $i\neq j$.
On the other hand, $\D_A\cap\D_B\neq \emptyset$.  We define the rectangular sets in $\bbR^d$:
 \begin{align}\label{geo10}
& Q_i:=[0,u_{k+1}]\times\dots\times [0,u_{k+d-1}]\times \bar I_j\,,\qquad j=1,\dots,2b,
\end{align}
and define the $\bbZ^d$ subsets
\begin{align}\label{geo20}
&A_i:=Q_{2i-1}\cap V\,,\quad B_i=Q_{2i}\cap V\,,\qquad i=1,\dots,b.\\&
A=\bigcup_{i=1}^bA_i\,,\qquad B=\bigcup_{i=1}^bB_i\,.
\end{align}
We refer to Figure \ref{fig2} for a two-dimensional representation. %\pietro{Aggiungere figura}
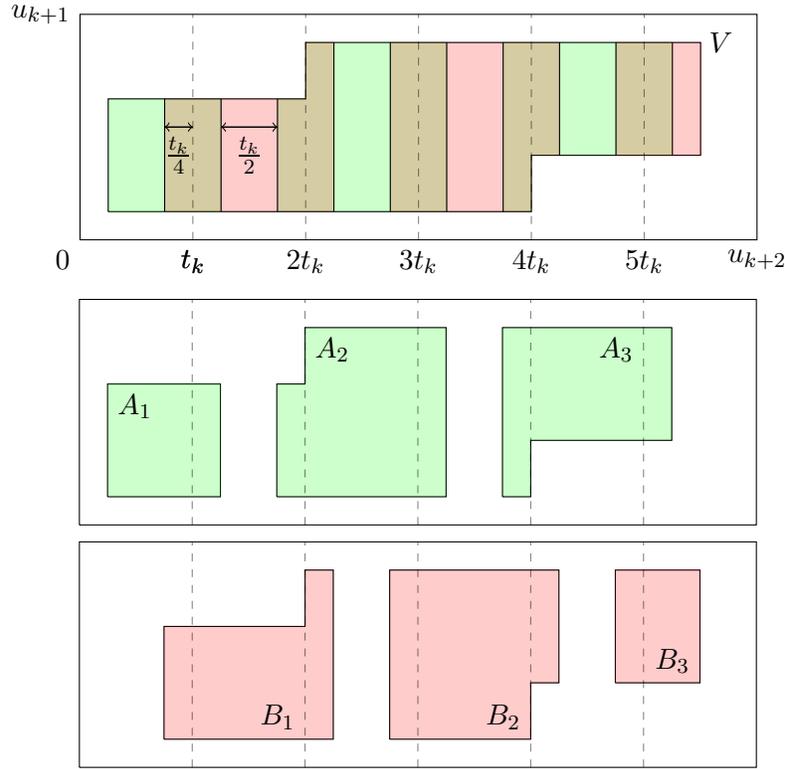
\begin{figure}[htp]
\center
\begin{subfigure}

\begin{tikzpicture}[scale=1.5]
\draw (0,0)--(6,0)--(6,2)--(0,2)--(0,0);
%\node [below left] {0};\node at (1,0) [below] {$t_{k}$};\node at (2,0) [below] {$2t_{k}$};\node at (3,0) [below] {$3t_{k}$};\node at (4,0) [below] {$4t_{k}$};\node at (5,0) [below] {$5t_{k}$};\node at (1,0) [below] {$t_{k}$};\node at (6,0) [below] {$u_{k+2}$};\node [left] at (0,2){$u_{k+1}$};
%\draw [fill=green, opacity=0.2] (0,0)--(5/4,0)--(5/4,2)--(0,2)--(0,0);\draw [fill=green, opacity=0.2] (7/4,0)--(13/4,0)--(13/4,2)--(7/4,2)--(7/4,0);\draw [fill=green, opacity=0.2] (15/4,0)--(21/4,0)--(21/4,2)--(15/4,2)--(15/4,0);
%\draw [fill=red, opacity=0.2] (3/4,0)--(9/4,0)--(9/4,2)--(3/4,2)--(3/4,0);\draw [fill=red, opacity=0.2] (11/4,0)--(17/4,0)--(17/4,2)--(11/4,2)--(11/4,0);\draw [fill=red, opacity=0.2] (19/4,0)--(6,0)--(6,2)--(19/4,2)--(19/4,0);
\draw (1/4,1/4)--(4,1/4)--(4,3/4)--(22/4,3/4)--(22/4,7/4)--(8/4,7/4)--(8/4,5/4)--(1/4,5/4)--(1/4,1/4);
\node [below left] {0};\node at (1,0) [below] {$t_{k}$};\node at (2,0) [below] {$2t_{k}$};\node at (3,0) [below] {$3t_{k}$};\node at (4,0) [below] {$4t_{k}$};\node at (5,0) [below] {$5t_{k}$};\node at (1,0) [below] {$t_{k}$};\node at (6,0) [below] {$u_{k+2}$};\node [left] at (0,2){$u_{k+1}$};
\draw [fill=green, opacity=0.2] (1/4,1/4)--(1/4,5/4)--(5/4,5/4)--(5/4,1/4)--(1/4,1/4);\draw [fill=green, opacity=0.2] (7/4,1/4)--(13/4,1/4)--(13/4,7/4)--(8/4,7/4)--(8/4,5/4)--(7/4,5/4)--(7/4,1/4);\draw [fill=green, opacity=0.2] (15/4,1/4)--(16/4,1/4)--(16/4,3/4)--(21/4,3/4)--(21/4,7/4)--(15/4,7/4)--(15/4,1/4);
\draw [fill=red, opacity=0.2] (3/4,1/4)--(9/4,1/4)--(9/4,7/4)--(8/4,7/4)--(8/4,5/4)--(3/4,5/4)--(3/4,1/4);\draw [fill=red, opacity=0.2] (11/4,1/4)--(16/4,1/4)--(16/4,3/4)--(17/4,3/4)--(17/4,7/4)--(11/4,7/4)--(11/4,1/4);\draw [fill=red, opacity=0.2] (19/4,3/4)--(22/4,3/4)--(22/4,7/4)--(19/4,7/4)--(19/4,3/4);
\draw [->] (3/4,1) -- (1,1);\draw [<-] (3/4,1)--(1,1);\node [below] at (7/8,1) {$\frac{t_{k}}{4}$};
\draw [->] (5/4,1) -- (7/4,1);\draw [<-] (5/4,1)--(7/4,1);\node [below] at (6/4,1) {$\frac{t_{k}}{2}$};
\draw[dashed,opacity=0.5](1,0)--(1,2);\draw(3/4,1/4)--(3/4,5/4);\draw(5/4,1/4)--(5/4,5/4);
\draw[dashed,opacity=0.5](2,0)--(2,2);\draw(7/4,1/4)--(7/4,5/4);\draw(9/4,1/4)--(9/4,7/4);
\draw[dashed,opacity=0.5](3,0)--(3,2);\draw(11/4,1/4)--(11/4,7/4);\draw(13/4,1/4)--(13/4,7/4);
\draw[dashed,opacity=0.5](4,0)--(4,2);\draw(15/4,1/4)--(15/4,7/4);\draw(17/4,3/4)--(17/4,7/4);
\draw[dashed,opacity=0.5](5,0)--(5,2);\draw(19/4,3/4)--(19/4,7/4);\draw(21/4,3/4)--(21/4,7/4);
\node [left] at (47/8,7/4) {$V$};
\end{tikzpicture}
\end{subfigure}

\begin{subfigure}

\hspace*{1em}
\begin{tikzpicture}[scale=1.5]
%\draw [fill=green, opacity=0.2] (1/4,1/4)--(1/4,5/4)--(5/4,5/4)--(5/4,1/4)--(1/4,1/4);\draw [fill=green, opacity=0.2] (7/4,1/4)--(13/4,1/4)--(13/4,7/4)--(8/4,7/4)--(8/4,5/4)--(7/4,5/4)--(7/4,1/4);\draw [fill=green, opacity=0.2] (15/4,1/4)--(16/4,1/4)--(16/4,3/4)--(21/4,3/4)--(21/4,7/4)--(15/4,7/4)--(15/4,1/4)
%\draw(5/4,0)--(5/4,2);
%\draw(7/4,0)--(7/4,2);
%\draw(13/4,0)--(13/4,2);
%\draw(15/4,0)--(15/4,2);
%\draw(21/4,0)--(21/4,2);
%\draw [fill=green, opacity = 0.2] (0,0)--(5/4,0)--(5/4,2)--(0,2)--(0,0);\draw [fill=green, opacity=0.2] (7/4,0)--(13/4,0)--(13/4,2)--(7/4,2)--(7/4,0);\draw [fill=green, opacity=0.2] (15/4,0)--(21/4,0)--(21/4,2)--(15/4,2)--(15/4,0);
\draw (0,0)--(6,0)--(6,2)--(0,2)--(0,0);
\draw (1/4,1/4)--(5/4,1/4)--(5/4,5/4)--(1/4,5/4)--(1/4,1/4);
\draw (7/4,1/4)--(13/4,1/4)--(13/4,7/4)--(8/4,7/4)--(8/4,5/4)--(7/4,5/4)--(7/4,1/4);
\draw (15/4,1/4)--(16/4,1/4)--(16/4,3/4)--(21/4,3/4)--(21/4,7/4)--(15/4,7/4)--(15/4,1/4);
\draw [fill=green, opacity=0.2] (1/4,1/4)--(1/4,5/4)--(5/4,5/4)--(5/4,1/4)--(1/4,1/4);\draw [fill=green, opacity=0.2] (7/4,1/4)--(13/4,1/4)--(13/4,7/4)--(8/4,7/4)--(8/4,5/4)--(7/4,5/4)--(7/4,1/4);\draw [fill=green, opacity=0.2] (15/4,1/4)--(16/4,1/4)--(16/4,3/4)--(21/4,3/4)--(21/4,7/4)--(15/4,7/4)--(15/4,1/4);
\draw[dashed,opacity=0.5](1,0)--(1,2);
\draw[dashed,opacity=0.5](2,0)--(2,2);
\draw[dashed,opacity=0.5](3,0)--(3,2);
\draw[dashed,opacity=0.5](4,0)--(4,2);
\draw[dashed,opacity=0.5](5,0)--(5,2);
\node [below right] at (1/4,5/4) {$A_{1}$};
\node [below right] at (8/4,7/4) {$A_{2}$};
\node [below left] at (20/4,7/4) {$A_{3}$};
\end{tikzpicture}
\end{subfigure}

\begin{subfigure}

\hspace*{1em}
\begin{tikzpicture}[scale=1.5]
%\draw(3/4,0)--(3/4,2);
%\draw(9/4,0)--(9/4,2);
%\draw(11/4,0)--(11/4,2);
%\draw(17/4,0)--(17/4,2);
%\draw(19/4,0)--(19/4,2);
%\draw [fill=red, opacity = 0.2] (3/4,0)--(9/4,0)--(9/4,2)--(3/4,2)--(3/4,0);\draw [fill=red, opacity=0.2] (11/4,0)--(17/4,0)--(17/4,2)--(11/4,2)--(11/4,0);\draw [fill=red, opacity=0.2] (19/4,0)--(6,0)--(6,2)--(19/4,2)--(19/4,0);
\draw (0,0)--(6,0)--(6,2)--(0,2)--(0,0);
\draw (3/4,1/4)--(9/4,1/4)--(9/4,7/4)--(8/4,7/4)--(8/4,5/4)--(3/4,5/4)--(3/4,1/4);\draw (11/4,1/4)--(16/4,1/4)--(16/4,3/4)--(17/4,3/4)--(17/4,7/4)--(11/4,7/4)--(11/4,1/4);\draw (19/4,3/4)--(22/4,3/4)--(22/4,7/4)--(19/4,7/4)--(19/4,3/4);
\draw [fill=red, opacity=0.2] (3/4,1/4)--(9/4,1/4)--(9/4,7/4)--(8/4,7/4)--(8/4,5/4)--(3/4,5/4)--(3/4,1/4);\draw [fill=red, opacity=0.2] (11/4,1/4)--(16/4,1/4)--(16/4,3/4)--(17/4,3/4)--(17/4,7/4)--(11/4,7/4)--(11/4,1/4);\draw [fill=red, opacity=0.2] (19/4,3/4)--(22/4,3/4)--(22/4,7/4)--(19/4,7/4)--(19/4,3/4);
\draw[dashed,opacity=0.5](1,0)--(1,2);
\draw[dashed,opacity=0.5](2,0)--(2,2);
\draw[dashed,opacity=0.5](3,0)--(3,2);
\draw[dashed,opacity=0.5](4,0)--(4,2);
\draw[dashed,opacity=0.5](5,0)--(5,2);
\node [above left] at (8/4,1/4) {$B_{1}$};
\node [above left] at (16/4,1/4) {$B_{2}$};
\node [above left] at (22/4,3/4) {$B_{3}$};
\end{tikzpicture}
\end{subfigure}

\caption{\footnotesize{An example of $A=\bigcup_{i}A_{i}$ (green blocks) and $B=\bigcup_{i}B_{i}$ (red blocks) for a given region $V$ in the rectangle $[0,u_{k+1}]\times[0,u_{k+2}].$}}
\label{fig2}

\end{figure}

We observe that $A_i \in \bbG_{k-1}$ and $B_i \in \bbG_{k-1}$ for all $i=1,\dots,b$. Indeed, the stretch of $A_i$ along the $d$-th coordinate is at most $t_k + 2t_k/4 \leq 2t_k\leq u_k$ which together with $u_{k,i}=u_{k-1,i+1}$, $i=1,\dots,d-1$, implies that $A_i \in \bbG_{k-1}$. The same applies to $B_i$. 
Observe that with these definitions one has the product property 
\begin{align}\label{geo21}
\mu_A=\otimes_{i=1}^b\mu_{A_i}\,,\qquad
\mu_B=\otimes_{i=1}^b\mu_{B_i}  .
\end{align}
Moreover, the geometric construction shows that 
\begin{align}\label{geo22}
d(V\setminus A,V\setminus B)\geq \frac12 t_k.
\end{align}
Thus, a repetition of the argument in Lemma \ref{lem:ABlemmad} shows that 
the assumption of Lemma \ref{lem:entBA} is satisfied with $\e$ given by 
$$\e_k = %O(u_{k+d}\times u_{k+d-1}^{d-1} e^{-at_k/2})=
O\left(u_{k}^d\, e^{-a u_k/2}\right).$$
Therefore, by Lemma \ref{lem:entBA},
\begin{align}\label{geo23}
\ent f\leq \mu\left[\ent_A f+ \ent_B f\right] + \theta(\e_k)\,\ent f.
\end{align}
Next, let $\r(k)$ be defined as the largest constant $\r > 0$ such that the inequality  
 \begin{equation}\label{EOth100}
\r\,\ent_V^\t f \leq \mu_V^\t\left[\ent_E f + \ent_O f\right]
 \end{equation}
holds for all $V\in\bbG_k$, $\t\in\O_{V^c}$, and all $f\geq 0$. 
The key observation is that thanks to the product property \eqref{geo21}, and using the fact that $A_i \in \bbG_{k-1}$ for all $i$, 
Lemma \ref{lem:parislemma} allows us to estimate
\begin{align}\label{geo24}
\r(k-1)\mu\left[\ent_A f\right]\leq \mu\left[\ent_{EA} f+ \ent_{OA} f\right].
\end{align}
Similarly,
\begin{align}\label{geo25}
\r(k-1)\mu\left[\ent_B f\right]\leq \mu\left[\ent_{EB} f+ \ent_{OB} f\right].
\end{align}
Thus, \eqref{geo23} implies 
\begin{align}\label{geo26}
\r(k-1)(1-\theta(\e_k))\ent f&\leq  \mu\left[\ent_{EA} f+ \ent_{OA} f\right] + \mu\left[\ent_{EB} f+ \ent_{OB} f\right]\\&
\leq 2 \mu\left[\ent_{E} f+ \ent_{O} f\right],
\end{align}
where we use the monotonicity of  $\L\mapsto \mu\left[\ent_{\L} f\right]$.
Estimating $1-\theta(\e_k)\geq 1/2$ we have proved that 
\begin{align}\label{geo27}
\r(k)\geq \frac14\r(k-1).
\end{align}
Iterating, we conclude $\r(k)\geq 4^{-k}\r(k_0)$. To finish the proof, observe that $(3/2)^k=b^{k\e}$ where $\e=\log(3/2)/\log (b)$, which can be made small by taking $b$ large. Therefore,
\begin{align}\label{geo270}
\d(k)\geq \r(\lfloor k\e\rfloor +1)\geq 4^{-k\e -1}\r(k_0)\geq c_04^{-k\e} = c_0 \ell_k^{-\e'},
\end{align}
where $c_0$ is a constant depending on $K,a,d,b$, while $\e'=d\log(4)/\log(b)$ can be as small as we wish provided $b$ is suitably large. 
This ends the proof of Theorem \ref{th:toy}.

\begin{remark}\label{rem:toy}
We point out that the argument given in the proof of Theorem \ref{th:toy} can be improved if one replaces the parameter $t_k$ which is linear in $u_k$ by $t'_k=C_1\log (u_k)$, with $C_1$ a suitably large constant. Since $t'_k$ is logarithmic in $u_k$, one can modify  the recursion to obtain a bound of the form $\d(k)\geq   \d(C_2\log(k))/C_2$ for some new constant $C_2$, which provides a much better lower bound on $\d(k)$ than the one stated in Theorem \ref{th:toy}. 
%cannot decrease to zero faster than any iterated logarithm of $k$ as $k\to\infty$. 
However, without the companion recursive estimate from Theorem \ref{th:maind}, this argument alone would not provide the uniform estimate $\inf_k\d(k)>0$.  
\end{remark}

\subsection{Proof of Theorem \ref{th:main} assuming $SM_L(K,a)$}
Theorem \ref{th:maind} and Theorem \ref{th:toy} allowed us to establish Theorem \ref{th:main} under the assumption $SM(K,a)$. We now prove it assuming only $SM_L(K,a)$. To this end we observe that any set $V\in\bbF^{(L)}$ is uniquely identified by the set $V'\in\bbF$ such that 
\begin{align}\label{relax0}
V = \bigcup_{y\in V'} Q_L(y).
\end{align} 
A careful check of the previous proofs then shows that if we work on the rescaled lattice, that is we replace vertices $x$ with blocks $Q_L(x)$, then we may repeat all steps in  Theorem \ref{th:maind} and Theorem \ref{th:toy} to obtain the following coarse-grained version of Theorem \ref{th:main} assuming  only $SM_L(K,a)$: for any $V\in \bbF^{(L)}$, for all $f\geq 0$,
\begin{align}\label{relax1}
\ent f \leq C\,\mu\left[\ent_{E_L}f + \ent_{O_L} f\right],
\end{align}
where, if $V$ is given by \eqref{relax0}, then $E_L=\cup_{x\in EV'}Q_L(x)$, and $ O_L=\cup_{x\in OV'}Q_L(x)$. 

Consider now a single cube $Q_L(x)$. By Lemma \ref{lem:base} we know that 
\begin{align}\label{relax2}
\ent_{Q_L(x)} f \leq C_1 \,\mu_{Q_L(x)}\left[\ent_{EQ_{L(x)}}f + \ent_{OQ_L(x)} f\right],
\end{align}
for some constant $C_1=C_1(L)$. 
Observe that by construction $d(Q_L(x),Q_L(y))>1$ for all $x,y\in EV'$. Similarly, $d(Q_L(x),Q_L(y))>1$ for all $x,y\in OV'$. Therefore, Lemma \ref{lem:parislemma} implies 
\begin{align}\label{relax3}
\ent_{E_L}f &\leq C_1 \,\mu_{E_L}\left[\ent_{EE_L}f + \ent_{OE_L} f\right]\\
\ent_{O_L}f &\leq C_1 \,\mu_{O_L}\left[\ent_{EO_L}f + \ent_{OE_L} f\right],
\end{align}
where $EE_L$ denotes the even sites in $E_L$, $EO_L$ the even sites in $O_L$, and so on.
Plugging these estimates in \eqref{relax1} and using the monotonicity of $A\mapsto\mu[\ent_A f]$ one arrives at  
\begin{align}\label{relax4}
\ent f \leq D\,\mu\left[\ent_{E}f + \ent_{O} f\right],
\end{align}
with $D=2C\times C_1$. This ends the proof of Theorem \ref{th:main}. 
\bibliographystyle{plain}
\bibliography{bibentropy}

 \end{document}